\documentclass[onefignum,onetabnum]{siamart250211}
\usepackage{amsmath, latexsym}
\usepackage{amssymb}
\usepackage[utf8]{inputenc}
\usepackage[english]{babel}
\usepackage{tikz, graphicx}
\usetikzlibrary{shapes.geometric, arrows}
\usetikzlibrary{positioning}
\usepackage{textpos}
\usepackage{amsfonts}
\usepackage{newlfont}
\usepackage{enumerate}
\usepackage{epsfig}
\usepackage{subfig}
\usepackage{subcaption}
\usepackage{float} 
\usepackage{multirow}
\usepackage{ntheorem}
\usepackage{mathrsfs}
\usepackage{xcolor}
\usepackage{bm}

\newsiamremark{remark}{Remark}

\def\L{\mathcal{L}}

\def\F{\mathcal{F}}

\def\BB{\mathfrak{B}}

\def\Y{\mathrm{Y}}

\def\G{\mathcal{G}}

\def\X{\mathbb{X}}

\def\U{\mathcal{U}}
\def\X{\mathcal{X}}
\def\Y{\mathcal{Y}}

\def\H{\mathcal{H}}
\def\pot{\partial\Omega}

\ifpdf
\hypersetup{
	pdftitle={An Example Article},
	pdfauthor={h}
}
\fi

\author{Arbaz Khan\thanks{Department of Mathematics, Indian Institute of Technology Roorkee, Roorkee 247667, India.
		E-mail: {arbaz@ma.iitr.ac.in}, {shiv\_m@ma.iitr.ac.in}.}
	\and Kent-Andr\' e Mardal\thanks{Department of Mathematics, University of Oslo, Norway.
		E-mail: {kent-and@simula.no}\\
		 \textbf{Funding:} ``Shiv Mishra is supported by UGC, India."
	}
	\and Shiv Mishra\footnotemark[1]
}

\ifpdf
\hypersetup{pdftitle={New \LaTeX\ Style} ,pdfauthor={S. Mishra and A. Khan}}
\fi
\begin{document}
	\title{Mixed Consistent PINNs for Elliptic Obstacle Problems with Stability Analysis}
\date{Date: \today}
\maketitle


\begin{abstract}
	We propose a consistent physics-informed neural networks (CPINNs) framework for elliptic obstacle problems formulated as variational inequalities. The method is based on a mixed loss functional that is rigorously aligned with the stability structure of the underlying problem and incorporates obstacle constraints through a consistent treatment of the associated Lagrange multiplier. Relying on optimal recovery theory under Besov regularity assumptions, we establish near-optimal convergence rates for the simultaneous reconstruction of the solution and the multiplier from pointwise interior and boundary data. To enable practical implementation, we construct discrete counterparts of the continuous stability norms and duality pairings, leading to fully computable and theoretically justified training losses. Numerical experiments on benchmark obstacle problems demonstrate the accuracy, stability, and robustness of the proposed approach, and highlight its clear advantages over standard PINNs.
\end{abstract}



\section{Introduction}
In this paper we consider the obstacle problem in primal form on a polygonal Lipschitz domain \(\Omega\subset \mathbb{R}^d, d\geq 2\) with boundary \(\partial\Omega\). The problem reads: Find $u\in H^1(\Omega)$ such that 
\begin{subequations}\label{1}
	\begin{align}
		- \Delta u &\geq f \quad \text{in } \Omega, \\
		u &\geq \psi \quad \text{in } \Omega, \\
		(u-\psi)\,(-\Delta u - f) &= 0 \quad \text{in } \Omega,\\
		u &= g \quad \text{on } \partial \Omega. 
	\end{align}
\end{subequations}
Here, $f$ is the source term, $\psi$ the obstacle function and $g$ the boundary condition, respectively. 
It is well-known from \cite{rodrigues1987obstacle} that for $f\in H^{-1}(\Omega)$, $\psi\in H^1(\Omega)$ and $g\in H^{1/2}(\partial\Omega)$ the problem admits a unique solution $u\in H^1(\Omega)$ and the following stability estimate holds
\begin{align}\label{mainestimate}
	c\, X \, \leq\|u\|_{H^1(\Omega)}\leq C\, X,\ \text{for}\ X:= \left(\|f\|_{H^{-1}(\Omega)}+\|g\|_{ H^{1/2}(\partial\Omega)} + \|\psi\|_{H^1(\Omega)} \right),
\end{align}
where the equivalence constants  $c,C>0$ depend only on $\Omega$.

The problem \eqref{1} translates well to a finite element setting for second order elliptic problems.   
Error estimates  were established in early works by~\cite{falk1974error, mosco1974one, brezzi1977error}, with comprehensive overviews provided in monographs such as~\cite{glowinski2013numerical}. An alternative line of research focuses on the Lagrange multiplier formulation, which introduces the reaction force as an additional unknown and naturally leads to mixed and stabilized finite element methods~\cite{gustafsson2017mixed, haslinger1996numerical}. This framework enables the use of semismooth Newton and active-set strategies, and has been analyzed in depth in~\cite{hintermuller2002primal}, including a priori and a posteriori error estimates for mixed and stabilized schemes~\cite{weiss2010posteriori,braess2005posteriori}, as well as Nitsche and penalty-type formulations~\cite{chouly2013nitsche, burman2017galerkin} for contact and obstacle problems. In~\cite{fuhrer2020first}, the authors provided refined stability results and improved discretization techniques by using the first-order least-squares method for the obstacle problem.

Recently, physics informed neural networks (PINNs)  have emerged as mesh-free solvers for PDEs~\cite{raissi2019physics}. Typically,  
the problem is phrased in a least square manner, which leads to increase regularity requirements~\cite{MSMR, zeinhofer2024unified}. 
The obstacle problem problem was considerd in such a setting in for instance \cite{el2025physics}, but a mathematical theory for error estimates
was not established.  
Alternatively, the 
deep Ritz method and finite neuron method~\cite{muller2022error, yu2018deep, xu2020finite} exploits standard regularity and may as such be extended to the obstacle problem in a similar manner as for standard finite elements. 
An interesting new alternative approach however is that of \cite{BVPS} where Besov space considerations enable the exploitation of bounds such as \eqref{mainestimate} although not from a variational setting, so-called consistent PINNs (CPINNs).  
The approach uses optimal recovery and Besov regularity to characterize the best-possible approximation rates for PDEs from pointwise data~\cite{BDS, krieg2022recovery}. 
The existing CPINN theory is however so far limited to standard PDEs like elliptic and parabolic~\cite{BVPS, mishra2025priorierroranalysisconsistent, https://doi.org/10.1002/nme.70320}, and does not cover variational inequalities or obstacle-type problems, where the presence of inequality constraints  requires a different treatment.

In our approach here we aim to mirror classical works~\cite{gustafsson2017mixed}, where mixed methods are exploited for the obstacle problem.  
However, as CPINNs is not expressed in a variational setting, but rather exploit Sobolev embeddings and non-standard formuations to
ensure $H^1$ convergence, care needs to be taken in the construction of the minimization problem. Specifically, we avoid the evaluation
of $H^{-1}$ type norms, which may be computationally expensive, while not requiring extra regularity as in standard PINNs.  

\subsection{Our contribution}
In this work, we propose a CPINNs framework for obstacle problems that encompasses both scalar and vector-valued formulations. In contrast to standard PINNs, the proposed approach systematically exploits stability estimates of the underlying problem through Besov-space techniques, in the spirit of CPINNs, thereby enabling a principled treatment of inequality constraints.

The main contributions of this work can be summarized as follows:
\begin{itemize}
	\item \textbf{Consistent loss formulation.} We introduce a mixed CPINNs loss functional that is equivalent to the natural stability norms of obstacle problems. The formulation incorporates obstacle constraints and associated Lagrange multipliers in a mathematically consistent and stable manner.
	
	\item \textbf{Extension of CPINNs to variational inequalities.} We extend the CPINNs framework beyond standard elliptic and parabolic equations to variational inequalities, thereby addressing problems with inequality constraints and intrinsically non-smooth solution structures.
	
	\item \textbf{Optimal recovery guarantees.} Under Besov regularity assumptions, we establish near-optimal convergence rates for the recovery of both the primal variable and the Lagrange multiplier from pointwise interior and boundary observations, using tools from optimal recovery theory.
	
	\item \textbf{Discrete norm equivalence and computable losses.} We construct discrete counterparts of the continuous stability norms and duality pairings, and prove their equivalence up to optimal recovery errors. This yields fully computable and theoretically justified training losses.
\end{itemize}

Numerical experiments on benchmark obstacle problems, including elasticity and higher-order models, validate the theoretical findings and demonstrate the improved accuracy and robustness of the proposed CPINNs approach compared to standard PINNs.

\subsection{Outline of the paper}
The paper is organized as follows. Section~\ref{Sec_2} introduces the model obstacle problem, its mixed formulation, and the associated functional setting and stability estimates. In Section~\ref{Sec_3}, we develop a mixed CPINNs loss formulation that incorporates obstacle constraints and Lagrange multipliers in a stability-equivalent manner. Section~\ref{Sec_4} contains an optimal recovery analysis under Besov regularity assumptions and establishes near-optimal convergence rates for both the solution and the multiplier from pointwise data. Section~\ref{Sec_5} constructs discrete counterparts of the continuous stability norms and duality pairings, which further yield fully computable training losses based on pointwise data. Section~\ref{Sec_6} reports numerical experiments on benchmark scalar, vector-valued, and higher-order obstacle problems, confirming the accuracy and robustness of the proposed CPINNs framework.

\section{Functional setting and optimization formulation}\label{Sec_2}
\subsection{Function spaces}
We use the standard notations for Sobolev spaces $H^1_0(\Omega)$ and $H^s(\Omega)$ for $s>0$. The trace space $H^{s-1/2}(\partial\Omega)$ is the space of boundary traces of functions from $H^s(\Omega)$, defined for all integers $s\geq 1.$ $H^{-1}(\Omega)$ is defined as the dual space corresponding to $H^{1}_0(\Omega)$, with norm denoted by $\|\cdot\|_{-1}$. The duality pairing $\langle \cdot, \cdot\rangle$ is understood as the continuous extension of the $L^2(\Omega)$ inner product. We remark that we often use vector functions denoted by boldface. The vectors are of size two and denoted as $\boldsymbol{v} = (v,\mu)$, and hence $(\cdot, \cdot)$ is used for listing the two functions.

Let us define space $W:= H^1_g(\Omega)=\{u\in H^1(\Omega)\mid u=g \text{ on } \partial\Omega\}$. We define the product space $V:= W\times H^{-1}(\Omega)$ with the corresponding norm 
$$\|\boldsymbol{v}\|_V^2:= \| v\|_{H^1(\Omega)}^2+\|\mu \|^2_{-1},\quad \text{for}\ \boldsymbol{v} = (v,\mu)\in V.$$

We further discuss the space $U = \{(u,\lambda)\in V\mid \, \Delta u+\lambda \in H^{-1}(\Omega) \}$, which is a closed space defined over a norm
$$\|\boldsymbol{u}\|_U^2:= \|u\|_{H^1(\Omega)}^2 +\|\Delta u +\lambda \|^2_{-1},\quad \text{for}\ \boldsymbol{u} = (u,\lambda)\in U.$$

For $v\in H^1_0(\Omega)$, we write $v\geq 0$, if $v\geq 0$ almost everywhere in $\Omega$. Similarly an element $\lambda\in H^{-1}(\Omega)$ is non-negative if $\langle \lambda, v\rangle \geq 0,$ for all $v\in  H^1_0(\Omega)$ with $v\geq 0.$ \\

It is easy to observe that the norm $\|\cdot\|_U$ is stronger than the norm $\|\cdot\|_V$. Indeed, for a positive $C$, we can show the estimate
\begin{align*}
	\|u\|_{H^1(\Omega)}^2+\|\lambda\|^2_{-1} &\leq \|u\|_{H^1(\Omega)}^2+\|\lambda+\Delta u\|^2_{-1} + \|\Delta u\|^2_{-1}\\
	&\leq \|u\|_{H^1(\Omega)}^2+\|\Delta u+\lambda\|^2_{-1} + \|\nabla u\|^2_{L^2(\Omega)}\\
	&\leq (1+C)\|u\|_{H^1(\Omega)}^2+\|\Delta u+\lambda\|^2_{-1} .
\end{align*}

\textbf{Besov spaces.}
Let $\Omega \subset \mathbb{R}^d$, $0<s<\infty$, and $0<p,q\le \infty$.
For an integer $r>s$ and a step size $h\in\mathbb{R}^d$, we denote the forward finite difference operator of order $r$ by $\Delta_h^r$.
For $t>0$, the modulus of smoothness of a function $f$ is defined by
\begin{align}
	\omega_r(f,t)_p := \sup_{|h|\le t} \|\Delta_h^r f\|_{L^p(\Omega_h)},\ \text{where}\ \Omega_h := \{x\in\Omega : [x,x+h]\subset\Omega\}.
\end{align}
Using this notation, the Besov space $B^s_{pq}(\Omega)$ is defined as
\begin{align}
	B^{s}_{pq}(\Omega) = \left\{ f\in W^{[s]}_{p}(\Omega)\mid  |f|_{B^{s}_{pq}} =  \Bigg( \int_{\Omega} [t^{-s}\omega_{r}(f,t)_p]^q \,\frac{dt}{t} \Bigg)^{\frac{1}{q}} < \infty \right\} ,
\end{align}
For the special case \( q' = \infty \), the Besov norm is computed by taking the supremum over \( t\in (0,h) \) instead of integrating with respect to \({dt}/{t} \). The corresponding function space is denoted by  $B^{s}_{p}(\Omega)$. The associated norm is given as
\[
\|f\|_{B^s_{pq}(\Omega)}
:= \|f\|_{W^{[s] }_p(\Omega)}
+ |f|_{B^s_{pq}(\Omega)}.
\]

\subsection{Minimization problem}
Let $u \in H^{1}(\Omega)$ denote the solution of the obstacle problem~\eqref{1}. Define the Lagrange multiplier (interpreted as the reaction force) \(\lambda = -\Delta u -f\) such that the obstacle problem~\eqref{1} can be reformulated into a second-order system, given as
\begin{align}\label{prob2}
	- \Delta u-\lambda &= f,  \quad 
	u \geq \psi,  \quad 
	\lambda\,(u-\psi) = 0, 
\end{align}
in domain \(\Omega\) with \(u|_{\partial\Omega} = g\). It is important to emphasize from~\cite{gustafsson2017mixed, fuhrer2024mixed} that under the conditions \(f\in H^{-1}(\Omega)\) and \(u\in H^{1}(\Omega)\), the Lagrange multiplier \(\lambda \in H^{-1}(\Omega)\).

Let \(u\in W (:=H_g^1(\Omega))\) be the unique solution of the obstacle problem. Let us formulate the functional 
\begin{align}\label{J(u)}
	J(u,\lambda) = \|\Delta u +\lambda + f\|^2_{-1} + \left < \lambda, u-\psi\,\right >,
\end{align}
where $\langle \cdot, \cdot\rangle$ denotes the duality pairing. Now, we develop the least-squares problem as a minimization over the admissible set
\begin{align*}
	K^s = \left\{(u,\lambda)\in W\times H^{-1}(\Omega)\mid u-\psi\geq 0, \lambda \geq 0\right\},
\end{align*}
which is a nonempty, convex, and closed set and, \(\psi\) is the given obstacle function. Thus, the corresponding minimization problem is to find \((u,\lambda)\in K^s\) such that
\begin{align}\label{minprob}
	J(u,\lambda) = \mathop{\min}_{(v,\mu)\in K^s} J(v,\mu).
\end{align}
The equivalence between the obstacle problem and a mixed least-squares formulation has been established in~\cite{fuhrer2020first} for obstacle problems with force term $f\in L^2(\Omega$), based on a first-order least-squares reformulation of the problem. The theorem stated below extends this framework to the natural setting $f\in H^{-1}(\Omega)$, as it provides the stability foundation for the consistent loss formulation used in the CPINNs framework.
\begin{theorem}\label{thm_2.1}
	Let \(f\in H^{-1}(\Omega)\), $\psi\in H^1(\Omega) $ and $g\in H^{1/2}(\partial\Omega)$. Then the problems \eqref{1} and \eqref{minprob} are equivalent. Moreover, the variational problem~\eqref{minprob} admits a unique solution $\boldsymbol{u}\in K^s$, and the following estimate holds
	\begin{align}\label{5}
		J(v,\mu) \geq C_J\, \|\boldsymbol{v}-\boldsymbol{u}\|^2_{U} \quad \text{for every } \boldsymbol{v}\in K^s,
	\end{align}
	where the constant $C_J$ depends only on domain $\Omega$.
\end{theorem}

\begin{proof}
	Let $\boldsymbol{u}:= (u,\lambda) = (u,- \Delta u-f)\in K^s$ be the unique solution of problem~\eqref{prob2}. By definition, we have $J(v,\mu)\geq 0$ for every $\boldsymbol{v} \in K^s$, and moreover, $J(u,\lambda)= 0$. Hence, $\boldsymbol{u}$ is a minimizer of the functional. \\
	Assume that condition \eqref{5} holds and $\boldsymbol{u}^*\in K^s$ is defined as another minimizer. Then from \eqref{5}, we get the condition that $\boldsymbol{u} = \boldsymbol{u}^*$. Thus, it remains only to verify inequality \eqref{5}. Let $\boldsymbol{v}:= (v,\mu)\in K^s$. Since each term in $J(v,\mu)$ is non-negative, and using the relation $f = - \Delta u-\lambda $. For a positive constant $C_p$, we obtain 
	\begin{align}\label{J_b}
		J(v,\mu) &= \|\Delta(v-u)+ (\mu-\lambda)\|^2_{-1} + \langle\mu, v-\psi\rangle\\
		&\geq \frac{1}{1+C_p^2} \left\{ (1+C_p^2)\|\Delta(v-u)+ (\mu-\lambda)\|^2_{-1} + \langle\mu, v-\psi\rangle \right\}.\notag
	\end{align}
	Furthermore, we recall the relations
	\begin{align*}
		\langle\lambda, u-\psi\rangle=0,\quad  \langle\lambda, v-\psi\rangle\geq 0,\quad  \langle\mu, u-\psi\rangle\geq 0.
	\end{align*}
	Therefore, for the last term, we compute
	\begin{align*}
		\langle\mu, v-\psi\rangle &= \langle\mu, v-u\rangle + \langle\mu, u-\psi\rangle + \langle\lambda, u-\psi\rangle\\
		&\geq \langle\mu, v-u\rangle + \langle\lambda, u-\psi\rangle + \langle\lambda, \psi-v\rangle\\
		&= \langle\mu, v-u\rangle + \langle\lambda,u-v\rangle = \langle\mu-\lambda, v-u\rangle.
	\end{align*}
	Let $\boldsymbol{w}:= \boldsymbol{v}-\boldsymbol{u} = (w,\nu).$ From~\eqref{J_b}, we get
	\begin{align*}
		J(v,\mu) &\geq  \frac{1}{1+C_p^2} \left\{ (1+C_p^2)\|\Delta(v-u)+ (\mu-\lambda)\|^2_{-1} + \langle\mu-\lambda, v-u\rangle\right\}\\
		&=  \frac{1}{1+C_p^2} \left\{  (1+C_p^2)\|\Delta w+ \nu\|^2_{-1}  + \langle\nu, w\rangle\right\}\\
		&=  \frac{1}{1+C_p^2} \left\{ (1+C_p^2)\|\Delta w+ \nu\|^2_{-1} + \langle\Delta w+ \nu, w\rangle-\langle\Delta w, w\rangle\right\}.
	\end{align*}
	Then, by using the Cauchy-Schwarz and Young's inequality, we obtain
	\begin{align*}
		J(v,\mu) &\geq  \frac{1}{1+C_p^2} \left\{ (1+C_p^2)\|\Delta w+ \nu\|^2_{-1} - C_P\|\Delta w+ \nu\|_{-1}\|\nabla w\|_{L^2(\Omega)}+\langle\nabla w, \nabla w\rangle\right\} \\
		&\geq  \frac{1}{1+C_p^2} \Bigg\{ (1+C_p^2)\|\Delta w+ \nu\|^2_{-1} - \left(\frac{C_p^2}{2}\|\Delta w+ \nu\|_{-1}^2+\frac{1}{2} \|\nabla w\|^2_{L^2(\Omega)}\right) +\|\nabla w\|^2_{L^2(\Omega)}\Bigg\}\\
		&\geq  \frac{1}{1+C_p^2} \Bigg\{ \left(1+\frac{C_p^2}{2}\right)\|\Delta w+ \nu\|^2_{-1} +\frac{1}{2}\|\nabla w\|^2_{L^2(\Omega)}\Bigg\}\\
		&\geq C_J\left(\|\Delta w+ \nu\|^2_{-1} +\|w\|^2_{H^1_0(\Omega)}\right)  \simeq \|\boldsymbol{w}\|^2_U = \|\boldsymbol{v}-\boldsymbol{u} \|^2_U  .
	\end{align*}
\end{proof}

\subsection{Loss formulation}
We consider a loss formulation that enforces both the PDE residual and the obstacle conditions. We now introduce the loss formulation derived from the \eqref{J(u)} by introducing the Lagrange multiplier \(\lambda\). Motivated by the structure of the functional \eqref{J(u)}, we define the mixed loss $\mathcal{L}_M$ over $K^s$ by
\begin{align}
	\mathcal{L}_M(\boldsymbol{v}) =  \|\Delta v +\lambda + f\|^2_{-1}+ \|g-v\|^2_{H^{1/2}(\partial\Omega)}  + \left < \lambda, v-\psi\,\right >.
\end{align}
Since we know that $J(u,\lambda)= 0$, the unique minimizer of the loss functional $	\mathcal{L}_M(\boldsymbol{v}) $ over the admissible convex set $K^s$ is the unique solution of the PDE~\eqref{prob2}. The corresponding minimization problem is formulated as
\begin{align*}
	u =\mathop{\arg\min}_{\boldsymbol{v}\, \in\, K^s} \mathcal{L}_M(\boldsymbol{v}).
\end{align*}

\section{Discretization method of PINNs}\label{Sec_3}
Neural networks have recently emerged as powerful tools for approximating solutions to partial differential equations (PDEs), including those described by equation~\eqref{1}. Consider a family of neural networks $\mathcal{N}_n$  with a fixed architecture containing $n$ parameters. The objective is to find an approximation $\hat{u}\in \mathcal{N}_n$ to the true solution $u$ by minimizing an appropriate norm $\|\cdot\|_X$.

To train the network, training data are generated from the domain $\Omega$ and its boundary $\partial\Omega$. 
These datasets, denoted by $f,\psi$ and $g$, are defined as $f_i = f(x_i)$, $\psi_i = \psi(x_i)$ and $g_j = g(x_j)$, where $\{x_i\}_{i=1}^m$ are points sampled in the interior and $ \{x_j\}_{j=1}^{\bar m}$ lie on the boundary. Collectively, these points form the sets \(\mathcal{X} := \{x_1,x_2,\dots, x_{\tilde m}\}, \,	\mathcal{Y} := \{x_1,x_2,\dots ,x_{\bar m}\}.\) To approximate the minimizer of the mixed loss introduced above, we employ corresponding PINNs framework.

\subsection{CPINNs framework using mixed loss}
To approximate the minimizer of the mixed loss $\mathcal{L}_M$, we extend the standard PINNs by introducing a neural approximation for the Lagrange multiplier $\lambda$. 
Let $\boldsymbol{\mathcal{N}}_n = \{(u_\theta,\lambda_\theta)\}$ denote a pair of neural networks with n parameters. In the mixed PINNs formulation, our objective is to find
\[
\boldsymbol{\hat u} = (\hat{u},\hat{\lambda}) =  \mathop{\arg\min}_{(v,\lambda) \in \boldsymbol{\mathcal{N}}_n} \L^p_M(v,\lambda),
\]
where
\begin{align*}
	\L^p_M(v,\lambda) =
	\frac{1}{\tilde{m}}\sum_{i=1}^{\tilde{m}}|\Delta v(x_i)+\lambda(x_i)+f(x_i)|^2
	+ \frac{1}{\bar{m}}\sum_{i=1}^{\bar{m}}|v(x_i)-g(x_i)|^2\\
	+ \frac{1}{\tilde{m}}\sum_{i=1}^{\tilde{m}}\left|\lambda(x_i)\,[v(x_i)-\psi(x_i)]\right|.
\end{align*}
This formulation enforces both the PDE and the obstacle constraint through the
coupling term $\langle \lambda, v-\psi\rangle$, which is evaluated over
least square norm discretization, and therefore does not correspond to the exact
discrete counterpart of the continuous stability norms.
We therefore employ the consistent PINNs framework, in which the continuous stability
norms appearing in the mixed formulation are replaced by their exact discrete
counterparts. To obtain a computable discretization of the $H^{-1}(\Omega)$ norm, we employ the Sobolev
embedding $L^\gamma(\Omega)\hookrightarrow H^{-1}(\Omega)$ with
$\gamma=2d/(d+2)$. The resulting consistent loss corresponding to $\mathcal{L}_M$ is given by
\begin{align}\label{discreteL}
	\L_M^c(v,\lambda) &=
	\left[\frac{1}{\tilde{m}}\sum_{i=1}^{\tilde{m}}|\Delta v(x_i)+\lambda(x_i)+f(x_i)|^\gamma\right]^{\frac{2}{\gamma}}
	+ \frac{1}{\bar{m}}\sum_{i=1}^{\bar{m}}\left|v(x_i)-g(x_i)\right|^2\\
	&+ \frac{1}{\bar{m}^2}\sum_{\substack{i,l=1 \\ i\neq l}}^{\bar{m}}
	\frac{\left|[v-g](x_i)-[v-g](x_l)\right|^2}{|x_i-x_l|^d} + \frac{1}{\tilde{m}}\sum_{i=1}^{\tilde{m}}\left|\lambda(x_i)\,[v(x_i)-\psi(x_i)]\right|.\notag
\end{align}

\section{Optimal recovery}\label{Sec_4}
In this section, we address the question of whether the available sampled data \( (f, \psi, g) \), observed at prescribed sets of locations \( \mathcal{X}, \mathcal{Y} \subset \overline{\Omega} \), provide sufficient information to reconstruct the solution \( u \) of \eqref{1} within a specified tolerance in the \( H^1(\Omega) \)-norm. This naturally leads to an optimal recovery (OR) framework, in which the attainable accuracy depends on the spatial configuration and number of sampling points, as well as on the regularity of the underlying data functions.

We assume that the inputs \( f, \psi,\) and \( g \) are continuous and belong to Besov spaces \( \mathfrak{B} \) compactly embedded in spaces of continuous functions. This assumption guarantees that the solution \( u \) of the boundary value problem is well-defined and lies in \( H^1(\Omega) \). Specifically, we consider
\begin{subequations}\label{3.24}
	\begin{align}
		f &\in \mathcal{F} := U(\BB), && \BB= B^s_{pq}(\Omega), \quad s > d/p, \ p, q \in [1,\infty], \\
		\psi &\in \mathcal{H} := U(\check{\BB}), && \check{\BB} = B^{\check{s}}_{\check{p}\check{q}}(\Omega), \quad \check{s} > d/\check{p}, \ \check{p} \in [1,2],\,\check{q} \in [1,\infty],\\
		g &\in \mathcal{G} := \mathrm{Tr}(U(\bar{\BB})), && \bar{\BB} = B^{\bar{s}}_{\bar{p}\bar{q}}(\partial\Omega), \quad \bar{s} > d/\bar{p}, \ \bar{p},\in [1,2], \, \bar{q} \in [1,\infty], 
	\end{align}
\end{subequations}
where \( U(\BB) \) denotes the unit ball in the corresponding Besov space. Under these regularity conditions, the embeddings  $\mathcal{F} \hookrightarrow C(\Omega) \subset H^{-1}(\Omega) , \
\mathcal{H} \hookrightarrow C(\Omega) \subset H^{1}(\Omega), $ and $
\mathcal{G} \hookrightarrow C(\partial\Omega) \subset H^{1/2}(\partial\Omega),$
are compact, which imply that the solution space
\[
\mathcal{U} := \big\{ u \in C(\Omega) \,\big|\, u \text{ solves } \eqref{1} \text{ for } f \in \mathcal{F}, \ \psi \in \mathcal{H}, \ g \in \mathcal{G} \big\}
\]
is a compact subset of \( H^1(\Omega) \). Given the observed samples \( f = (f_1, \ldots, f_{\tilde m}) \), \( \psi = (\psi_1, \ldots, \psi_{\tilde m}) \), and \( g = (g_1, \ldots, g_{\bar m}) \), the accessible information about the underlying functions is restricted to the following data-consistent sets:
\begin{align*}
	\mathcal{F}_{\mathrm{data}} &:= \{ f \in \mathcal{F} : f_i = \hat{f}_i, \ i = 1, \ldots, \tilde m \}, \\
	\mathcal{H}_{\mathrm{data}} &:= \{ \psi \in \mathcal{H} : \psi_j = \hat{\psi}_j, \ j = 1, \ldots, \tilde m \},  \\
	\mathcal{G}_{\mathrm{data}} &:= \{ g \in \mathcal{G} : g_j = \hat{g}_j, \ j = 1, \ldots, \bar m \}.
\end{align*}
Accordingly, the feasible set of solutions consistent with the measurements is defined as
\begin{align*}
	\mathcal{U}_{\text{data}} &:= \left\{ u \in \mathcal{U} \, \middle| \, 
	\begin{aligned}
		\Delta u(x_i) \geq \hat{f}_i,\ &u(x_i) \geq \hat \psi_i,\ && i = 1, \dots, \tilde m, \\
		\hspace{-2mm}u(x_j) &= \hat{g}_j, && j = 1, \dots, \bar m, \\
		(\Delta u(x_i)-f(x_i))&(u(x_i)-\psi(x_i)) = 0 ,\ && i = 1,\dots,\tilde m.\\
	\end{aligned}
	\right\},
\end{align*}
or equivalently,
\[
\mathcal{U}_{\mathrm{data}}
= \big\{ u : u \text{ solves } \eqref{1} \text{ with } f \in \mathcal{F}_{\mathrm{data}}, \ \psi \in \mathcal{H}_{\mathrm{data}}, \ g \in \mathcal{G}_{\mathrm{data}}\big\}.
\]

The central question is how strongly the constraint \( u \in \mathcal{U}_{\mathrm{data}} \) determines the true solution. The fidelity of the recovered approximation depends on both the geometry of the sampling sets \( \mathcal{X} \) and \( \mathcal{Y} \) and the specific measured values. To characterize the best achievable accuracy under fixed sampling configurations, we define the uniform optimal recovery rate by
\[
R^*(\mathcal{U}, \mathcal{X}, \mathcal{Y})_X := 
\sup_{u \in \mathcal{U}} R\big(\mathcal{U}_{\mathrm{data}}(u)\big)_X,
\]
which measures the worst-case recovery error in a given norm \( \|\cdot\|_X \) over all admissible solutions \( u \in \mathcal{U} \). The optimal rate achievable using at most \( m \) data samples is then defined as
\[
R^*_m(\mathcal{U})_X :=
\inf_{\substack{\mathcal{X} \subset \Omega,\, \mathcal{Y} \subset \partial\Omega \\ |\mathcal{X}| + |\mathcal{Y}| = m}}
R^*(\mathcal{U}, \mathcal{X}, \mathcal{Y})_X, \quad m \ge 2.
\]
Analogously, for the recovery of the forcing term \( f \in \mathcal{F} \), we define
\[
R^*(\mathcal{F}, \mathcal{X})_X :=
\sup_{f \in \mathcal{F}} R\big(\mathcal{F}_{\mathrm{data}}(f)\big)_X, 
\quad
R^*_m(\mathcal{F})_X :=
\inf_{\mathcal{X} \subset \Omega,\, |\mathcal{X}| = \tilde m}
R^*(\mathcal{F}, \mathcal{X})_X,
\]
and similar expressions will work for the boundary data \( g \) and obstacle data \( \psi \). The recovery performance in each case depends explicitly on the sampling configuration and the norm used to quantify the error.

\subsection{Optimal recovery of $u$}

We now analyze the information-theoretic recovery limits for both formulations introduced earlier.
The first corresponds to the scalar solution $u$ of problem~\eqref{1}, whereas the second concerns the coupled pair ($u,\lambda)$ solving the system~\eqref{prob2}.

\subsubsection{Optimal recovery of the primal solution $u$}
This section investigates the best achievable accuracy for approximating solutions \( u \in \mathcal{U} \) within the Sobolev space \( H^1(\Omega) \). Consider the boundary value problem \eqref{1}, where the available information is given through a total of \( m = \tilde{m} + \bar{m} \) sampling points. We assume that the smoothness parameters satisfy \( s > d/p \), \( \bar{s} > d/\bar{p} \), and \( \check{s} > d/\check{p} \) which ensures that the corresponding Besov spaces are continuously embedded into spaces of continuous functions. On the basis of these smoothness conditions, we introduce the following data classes, \( \F := U(B^s_{pq}(\Omega)) \), \( \G := \operatorname{Tr}(U(B^{\bar{s}}_{\bar{p}\bar{q}}(\Omega))) \), and \( \H :=  U(B^{\check{s}}_{\check{p}\check{q}}(\Omega))\). Consequently, the family of admissible solutions is defined as
\[
\U := \left\{ u \in C(\Omega) : u \text{ satisfies \eqref{1} for some  } f \in \F,\, g \in \G,\, \psi\in \H \right\}.
\]
We now present a theorem that characterizes the optimal recovery rate associated with the class \( \mathcal{U} \). For simplicity of exposition and without loss of generality, we present the optimal recovery analysis on the unit cube $\Omega = (0,1)^d$, the results extend to general bounded Lipschitz domains by standard localization arguments.
\begin{theorem}
	Let \( \Omega = (0,1)^d  \), and let  \( \F, \G, \H \), be the data classes and \(\U \) the corresponding solution class defined above. Then the optimal approximation rate for recovering \(u\in \U\) in the \( H^1(\Omega) \) norm satisfies
	\begin{align}
		R^*_m(\U)_{ H^1(\Omega)} \asymp m^{-\min(a,b,c)},
	\end{align}
	where $a,b,$ and $c$ denote the optimal recovery exponents associated with $f,g,$ and $h,$ respectively.
\end{theorem}
\begin{proof}
	Let \( \X\) and \(\Y\) be the sets of sampling locations for the interior and boundary data, respectively, with cardinalities \( |\X| = \tilde m \) and \( |\Y| = \bar m \), so that the total number of location points \( m = \tilde m + \bar m \). For any pair of solutions \( u_1, u_2 \in \U_{\text{data}} \) corresponding to available data values \( f_i \in \F_{\text{data}},\, g_i \in \G_{\text{data}},\, \psi_i \in \H_{\text{data}} \), the stability bound \eqref{mainestimate} implies
	\begin{align}\label{3.23}
		\|u_1 - u_2\|_{H^1(\Omega)} \asymp \|f_1 - f_2\|_{-1}+ \|g_1 - g_2\|_{H^{1/2}(\partial\Omega)} + \|\psi_1-\psi_2\|_{H^{1}(\Omega)}.
	\end{align}
	Hence, the optimal recovery performance for the entire class \(\U\) can be related to that of the data spaces
	\begin{align}
		R^*(\U, \X, \Y)_{H^1(\Omega)} \asymp R^*(\F, \X)_{H^{-1}(\Omega)} + R^*(\G, \Y)_{H^{1/2}(\pot)} + R^*(\H, \X)_{H^{1}(\Omega)}.
	\end{align}
	Let us assume the sample budget is distributed comparably between the interior and boundary measurements, i.e., \( \tilde m \asymp \bar m \), we use~\cite[Theorems~3.1, 3.2]{BVPS}.
	\begin{align*}
		R^*_m(\U)_{H^1(\Omega)} \lesssim \tilde m^{-a} + \bar m^{-b} + \bar m^{-c} \lesssim m^{-\min(a, b, c)}. 
	\end{align*}
	Conversely, from the same theorems we derive
	\begin{align*}
		m^{-\min{(a,b,c)}}&\lesssim \inf_{m = \tilde{m}+ \bar{m}} \left(\tilde{m}^{-a} + \bar{m}^{-b} + \bar{m}^{-c} \right)\\
		&\lesssim \inf_{m = |\X| + |\Y|} R^*(\U, \X, \Y)_{H^1(\Omega)} = R^*_m(\U)_{H^1(\Omega)},
	\end{align*}
	which yields the claimed equivalence and completes the argument.
\end{proof}

\subsubsection{Optimal recovery of the mixed pair $(u,\lambda)$}
We now extend the optimal recovery analysis to the mixed formulation of the obstacle problem defined in~\eqref{prob2}. Let us define the solution space
\begin{align}
	\U_M:= \{(u,\lambda)\in W\times H^{-1}(\Omega): (u,\lambda) \text{ satisfy } \eqref{prob2} \text{ for some } f\in \F, g\in \G, \psi \in \H  \}
\end{align}
Under the assumptions \( s > d/p \), \( \bar{s} > d/\bar{p} \), and \( \check{s} > d/\check{p} \) the embeddings 
$\mathcal{F} \hookrightarrow C(\Omega) \subset H^{-1}(\Omega) , \
\mathcal{H} \hookrightarrow C(\Omega) \subset H^{1}(\Omega), $ and $
\mathcal{G} \hookrightarrow C(\partial\Omega) \subset H^{1/2}(\partial\Omega),$
are compact, which ensure that $\U_M$ is compact in the product norm
\[\|(u,\lambda)\|^2_{U} = \|u\|^2_{H^1(\Omega)}+\|\Delta u+\lambda\|^2_{-1}. \]
For each admissible sampling configuration $\X\subset\Omega$ and $\Y\subset\pot$ with cardinalities $|\X| = \tilde m$ and $|\Y| = \bar m$ we define the feasible data sets $\F_{\text{data}},\, \G_{\text{data}},\, \H_{\text{data}},$
as in the previous section, and the corresponding feasible solution class
\begin{align*}
	\mathcal{U}_{M,\text{data}} := \{ (u,\lambda)\in \mathcal{U}_M : u,\lambda \text{ consistent with } (\F_{\text{data}},\G_{\text{data}},\H_{\text{data}})\}.
\end{align*}
From the stability estimate~\eqref{5} we have, for any $(u_1,\lambda_1), (u_2,\lambda_2) \in \mathcal{U}_{M,\text{data}}$, the following estimate holds
\begin{align}
	\|(u_1,\lambda_1)- (u_2,\lambda_2) \|_U\lesssim \|f_1-f_2\|_{-1}+\|g_1-g_2\|_{H^{1/2}(\pot)} + \|\psi_1-\psi_2\|_{H^1(\Omega)}.
\end{align} 

\begin{theorem}
	Let $\Omega \subset \mathbb{R}^d$ be a bounded domain. Let $(u_1, \lambda_1) \in U$ be the solution to the obstacle problem corresponding to data $(f_1, g_1, \psi_1)$ and $(u_2, \lambda_2) \in U$ be the solution corresponding to data $(f_2, g_2, \psi_2)$.
	For the given solution space $U$, the following estimate holds
	\begin{align}\label{main3.1}
		\|(u_1,\lambda_1)- (u_2,\lambda_2) \|_U\lesssim \|f_1-f_2\|_{-1}+\|g_1-g_2\|_{H^{1/2}(\pot)} + \|\psi_1-\psi_2\|_{H^1(\Omega)}.
	\end{align} 
\end{theorem}
\begin{proof}
	We want to bound the norm of the difference $\|{(u_1, \lambda_1) - (u_2, \lambda_2)}\|_U^2$. By definition of the norm, we have
	\begin{equation} \label{unorm_def}
		\|{(u_1, \lambda_1) - (u_2, \lambda_2)}\|_U^2 = \|{u_1 - u_2}\|_{H^1(\Omega)}^2 + \|{\Delta(u_1 - u_2) + (\lambda_1 - \lambda_2)}\|_{-1}^2
	\end{equation}
	We proceed by bounding the two terms on the right-hand side separately. As shown in~\cite[Chapter~I, Section~3]{haslinger1996numerical}, if $(u,\lambda)$ is the solution of problem \ref{prob2}, then the corresponding $u$ is the solution of \eqref{1}. Thus, from \eqref{3.23}, we get the upper bound for the norm
	\begin{align}\label{thm3.1_1}
		\|u_1 - u_2\|_{H^1(\Omega)} \leq C_P\left(\|f_1 - f_2\|_{-1} + \|g_1 - g_2\|_{H^{1/2}(\partial\Omega)} + \|\psi_1-\psi_2\|_{H^{1}(\Omega)}\right).
	\end{align}
	Since both the solutions $(u_1, \lambda_1)$ and $(u_2, \lambda_2)$ satisfies \eqref{prob2}. Thus we get
	\[\Delta u_1 +\lambda_1 = f_1 \ \text{and} \ \Delta u_2 +\lambda_2 = f_2. \]
	Therefore, we get the result from the above equations
	\begin{align}
		\Delta (u_1- u_2 ) +\lambda_1 -\lambda_2 = f_1 -f_2.
	\end{align}
	Now after taking the $H^{-1}(\Omega)$ norm both sides, we get
	\begin{align}\label{thm3.1_2}
		\|\Delta (u_1- u_2 ) +(\lambda_1 -\lambda_2)\|_{-1} = \|f_1 -f_2\|_{-1}
	\end{align}
	After putting the estimates \eqref{thm3.1_1} and \eqref{thm3.1_2} in \eqref{unorm_def}, we get
	\begin{align*}
		\|{(u_1, \lambda_1) - (u_2, \lambda_2)}\|_U^2 &\leq C^2_P\left(\|f_1 - f_2\|^2_{-1}+ \|g_1 - g_2\|^2_{H^{1/2}(\partial\Omega)} + \|\psi_1-\psi_2\|^2_{H^{1}(\Omega)}\right) \\&\hspace{4cm}+ \|f_1 - f_2\|^2_{-1}\\
		&\hspace{-0.2cm}\leq C \left(\|f_1 - f_2\|^2_{-1} + \|g_1 - g_2\|^2_{H^{1/2}(\partial\Omega)} + \|\psi_1-\psi_2\|^2_{H^{1}(\Omega)}\right).
	\end{align*}
	Thus, we get the final stability estimate
	\begin{align*}
		\|{(u_1, \lambda_1) - (u_2, \lambda_2)}\|_U \lesssim \left(\|f_1 - f_2\|_{-1} + \|g_1 - g_2\|_{H^{1/2}(\partial\Omega)} + \|\psi_1-\psi_2\|_{H^{1}(\Omega)}\right).
	\end{align*}
\end{proof}
Consequently, the optimal recovery rate for the mixed pair satisfies 	
\begin{align}
	R^*_m(\U)_{ H^1(\Omega)} \asymp m^{-\min(a,b,c)},
\end{align}
where $a,b,c$ denote the recovery exponents associated with $f,g,\psi,$ respectively.
The rate coincides with that of the scalar formulation, reflecting that both $u$ and $(u,\lambda)$ share identical smoothness conditions from the data spaces $\F,\G,\H$.

\section{Norm discretization}\label{Sec_5}
To obtain a fully computable loss functional, we replace the continuous norms appearing in the theoretical loss by discrete counterparts constructed from pointwise evaluations at collocation points. These discrete norms are shown to be equivalent to the corresponding continuous norms over the admissible Besov classes, up to the optimal recovery rate.
\subsection{$L^\tau(\Omega)$ norm discretization}\label{sec_5.1}
Let $\Omega = (0,1)^d$ and let $f\in B^s_{pq}(\Omega)$ with $s>d/p$. Consider a uniform grid $G_{k,r}\subset \Omega$ for $\tilde m = |G_{k,r}| \simeq (r2^k)^d$ and $r>\max\{s,1\}$. Define the discrete $L^\tau$ norm for $\tau\in[1,\infty)$by
\begin{align}\label{4.1}
	\|f\|^*_{L^\tau} =\left[ \frac{1}{\tilde{m}}\sum_{i=1}^{\tilde m} |f(x_i)|^\tau\right]^{\frac{1}{\tau}},
\end{align}
with supremum norm for $\tau=\infty$. Let $S_k^*(f)$ denote the simplicial polynomial interpolant of $f$ based on the grid $G_{k,r}$. From \cite{BVPS}, operator $S_k^*(f)$ achieves the optimal recovery rate in $L^\tau$ as
\begin{align}\label{L1}
	\|f-S_k^*(f)\|_{L^\tau(\Omega)} \lesssim \|f\|_{B^s_{pq}(\Omega)} \tilde{m}^{-\alpha_{\tau}},\quad \alpha_\tau = \frac{s}{d}-\left(\frac{1}{p}-\frac{1}{\tau}\right)_+.
\end{align}
Moreover, when evaluated at discrete data sites, the norm of the simplicial interpolant coincides with the norm defined over $f$ given by $\|S_k^*(f)\|_{L^\tau(\Omega)} \simeq \|f\|_{L^\tau(\Omega)}^{*}.$ Combining these estimates yields the norm equivalence
\begin{align*}
	\|f\|_{L^\tau(\Omega)} \lesssim \|f\|_{L^\tau(\Omega)}^{*}
	+ \|f\|_{B^s_{pq}}\,\tilde m^{-\alpha_\tau},
	\quad
	\|f\|_{L^\tau(\Omega)}^{*} \lesssim \|f\|_{L^\tau(\Omega)}
	+ \|f\|_{B^s_{pq}}\,\tilde m^{-\alpha_\tau}.
\end{align*}
Thus, the discrete norm approximates the continuous $L^\tau(\Omega)$ norm up to the optimal recovery error.
\begin{remark}\label{rm1}
	In a similar way, one can construct a discrete approximation of the $H^{1}(\Omega)$ norm based on simplicial interpolation.  
	If $f \in B^{s}_{pq}(\Omega)$ with $s > d/p$ and $p\in(0,2]$, then the corresponding discrete norm $\|f\|_{H^{1}(\Omega)}^{*}$ satisfies
	\begin{align}\label{H1}
		\|f\|_{H^{1}(\Omega)} \;\simeq\; \|f\|_{H^{1}(\Omega)}^{*}
		\;+\; \|f\|_{B^{s}_{pq}}\,\tilde m^{-\alpha{'}},
		\quad
		\alpha{'} := \frac{s-1}{d} - \Bigl(\frac{1}{p}-\frac{1}{2}\Bigr)_+ .
	\end{align}
\end{remark}
\subsection{$H^{1/2}(\partial\Omega)$ norm discretization}
In this section, we introduce the discrete $H^{1/2}(\partial\Omega)$ norm and analyze their accuracy in approximating the continuous $H^{1/2}(\partial\Omega)$ norm for functions $g$ belonging to the model class $\G = \operatorname{Tr}(U(\bar \BB))$. The discrete norm is constructed solely from point evaluations of $g$
at boundary data sites and is therefore computable from the given data.

Let $\Y = \{x_j\}_{j=1}^{\bar m}$ be the boundary data sites, with $\bar m \simeq 2^{k(d-1)}$. For a function $g$ defined on $\Y$, the discrete $H^{1/2}(\partial\Omega)$ seminorm is defined by the averaged pairwise difference quotient
\begin{align}
	|g|^*_{H^{1/2}(\partial\Omega)} := \left(\frac{1}{\bar m^2} \sum_{i\neq j}^{}\frac{|g(x_i)-g(x_j)|^2}{|x_i-x_j|^d}\right)^{\frac{1}{2}}.
\end{align}
The corresponding discrete $H^{1/2}(\pot)$ norm is given by
\begin{align}
	\|g\|^*_{H^{1/2}(\pot)} := \|g\|^*_{L^2(\pot)} + |g|^*_{H^{1/2}(\pot)},
\end{align}
where discrete $L^2(\pot)$ is defined as in \eqref{4.1} for $\tau=2$. Next, theorem shows that the discrete $H^{1/2}(\pot)$ norm approximates the continuous norm up to the optimal recovery rate.

\begin{theorem}\cite[Theorem~6.4]{BVPS}
	Let $\G = \operatorname{Tr}(U(\bar{\BB}))$, where $\bar{\BB} = B^{\bar s}_{\bar p\bar q}(\Omega)$ with $\bar s>d/2$, and let $\beta = \bar s-1/d-1$. Then for each $g\in\G$ and any $\bar m\geq 1$, the continuous and discrete $H^{1/2}(\pot)$ norms satisfy
	\begin{align*}
		\|g\|_{H^{1/2}(\pot)} &\lesssim \|g\|_{H^{1/2}(\pot)}^{*}
		+ \|g\|_{\bar{\BB}}\,\bar m^{-\beta},
		\\
		\|g\|_{H^{1/2}(\pot)}^{*} &\lesssim \|g\|_{H^{1/2}(\pot)}
		+ \|g\|_{\bar{\BB}}\,\bar m^{-\beta}.
	\end{align*}
\end{theorem}

\subsection{Discrete approximation of the duality pairing}\label{sec_5.3}
The obstacle formulation involves the duality pairing
$\langle \lambda, u-\psi\rangle$, where
$\lambda \in H^{-1}(\Omega)$ and $(u-\psi)\in H^{1}(\Omega)$.
To obtain a fully computable loss functional, we construct a discrete
approximation of this pairing using only pointwise evaluations at interior
collocation points.

Let $X = \{x_i\}_{i=1}^{\tilde m} = G_{k,r} \subset \Omega$ be the interior grid
with $\tilde m \simeq (r2^k)^d$.
We define the discrete duality pairing by
\begin{align}\label{eq:discrete_duality_omega}
	\langle \lambda, u-\psi\rangle^{*}
	:= \frac{1}{\tilde m}\sum_{i=1}^{\tilde m}
	\lambda(x_i)\,(u(x_i)-\psi(x_i)).
\end{align}
This definition depends only on pointwise evaluations of $\lambda$, $u$, and
$\psi$ at the collocation points and is therefore directly computable from the
available data.
Using the equivalence of discrete $H^{1}(\Omega)$ norm established in the Remark~\ref{rm1} and the discrete $L^{\tau}(\Omega)$ norm defined in \eqref{4.1}, together with the Sobolev embedding $L^\tau(\Omega)\hookrightarrow H^{-1}(\Omega)$ for $\tau= 2d/(d+2)$, the discrete pairing \eqref{eq:discrete_duality_omega} can be viewed as
a quadrature approximation of the continuous duality product.  

	We now establish the relation between the continuous and discrete pairings. From H$\ddot{o}$lder’s inequality we get, 
	\begin{align}\label{5.5}
		|\langle \lambda, u-\psi\rangle|
		\le \|\lambda\|_{L^\tau(\Omega)}\,\|u-\psi\|_{L^{\tau'}(\Omega)}\le \|\lambda\|_{L^\tau(\Omega)}\,\|u-\psi\|_{H^1(\Omega)}.
	\end{align}
	Using H$\ddot{o}$lder’s inequality for the discrete norms, we obtain
	\begin{align*}
		|\langle \lambda, u-\psi\rangle^*|
		\le \|\lambda\|^*_{L^\tau(\Omega)}\,\|u-\psi\|^*_{L^{\tau'}(\Omega)}\le  \|\lambda\|^*_{L^\tau(\Omega)}\,\|u-\psi\|^*_{H^1(\Omega)}.
	\end{align*}
	where the discrete norms are defined as in Section~\ref{sec_5.1}. By the discrete-continuous norm equivalence for Besov-regular functions, it holds that
	\begin{align*}
		\|\lambda\|^*_{L^\tau(\Omega)} \lesssim \|\lambda\|_{L^\tau(\Omega)}
		+ \|\lambda\|_{B^s_{pq}(\Omega)}\,\tilde m^{-\alpha_\tau},
	\end{align*}
	and similarly for $\check{\BB} = B^{\check{s}}_{\check{p}\check{q}}(\Omega)$ with $\check{s} >d/\check{p}$, we get
	\begin{align*}
		\|u-\psi\|^{*}_{H^1(\Omega)} \lesssim \|u-\psi\|_{H^1(\Omega)}
		+ \|u-\psi\|_{B^{\check{s}}_{\check{p}\check{q}}(\Omega)}\,\tilde m^{-{\alpha'}},
	\end{align*}
	where $\alpha_\tau$ and $\alpha'$ are defined in \eqref{L1} and \eqref{H1} respectively. Combining the above estimates, we obtain
	\begin{align}\label{5.6}
		|\langle \lambda, u-\psi\rangle^*|\lesssim 	\|\lambda\|_{L^\tau(\Omega)} 	\|u-\psi\|_{H^1(\Omega)} + \Big(\|\lambda\|_{L^\tau(\Omega)}\|u-\psi\|_{\check{\BB}} + \|\lambda\|_{{\BB}}\|u-\psi\|_{H^1(\Omega)}\\\notag + \|\lambda\|_{{\BB}}\|u-\psi\|_{\check{\BB}}\Big) \tilde m^{-\min\{\alpha_\tau, \alpha'\}}.
	\end{align}
	After subtracting \eqref{5.5} and \eqref{5.6} and applying the triangle inequality, we obtain
	\begin{align}\label{5.7}
		\big|	|\langle \lambda, u-\psi\rangle|-	|\langle \lambda, u-\psi\rangle^* |\big| &\leq 	\Big(\|\lambda\|_{L^\tau(\Omega)}\|u-\psi\|_{\check{\BB}} + \|\lambda\|_{{\BB}}\|u-\psi\|_{H^1(\Omega)}\\&\notag\hspace{3cm} + \|\lambda\|_{{\BB}}\|u-\psi\|_{\check{\BB}}\Big) \tilde m^{-\min\{\alpha_\tau, \alpha'\}}\\
		&\lesssim\Big(\|\lambda\|_{L^\tau(\Omega)}^2+\|u-\psi\|^2_{\check{\BB}} + \|\lambda\|_{{\BB}}^2+\|u-\psi\|^2_{H^1(\Omega)}\notag \\&\notag\hspace{3cm} + \|\lambda\|_{{\BB}}^2+\|u-\psi\|^2_{\check{\BB}}\Big) \tilde m^{-\min\{\alpha_\tau,\alpha'\}}.\notag
	\end{align}
	For the given conditions $s>d/p$ and $\check s>d/\check p$, the corresponding Sobolev-Besov embeddings $B^s_{pq}(\Omega) \hookrightarrow L^\tau(\Omega)$ and $B^{\check{s}}_{\check{p}\check{q}}(\Omega) \hookrightarrow H^1(\Omega) $ holds respectively. Thus, from \eqref{5.7}, we get
	\begin{align}\label{5.8}
		\big|	|\langle \lambda, u-\psi\rangle|-	|\langle \lambda, u-\psi\rangle^* |\big| 
		\lesssim\left( \|\lambda\|_{{\BB}}^2+\|u-\psi\|^2_{\check{\BB}}\right) \tilde m^{-\min\{\alpha_\tau,\alpha'\}}.
	\end{align}
	Therefore, the discrete duality pairing provides a consistent approximation of the continuous pairing, with an error controlled by the Besov norms of $\lambda$ and $u-\psi$.
	
	\subsection{Discrete residual control}
	To render the theoretical loss functional \( \L_M \), which involves continuous norms, practically computable, we replace it with a discrete analogue \( \L^c_M\). This surrogate relies solely on the sampled evaluations of the data functions \( (f, g, \psi) \). Let \( k\) and \(r\) denote the discretization parameters. We define the sets of interior sampling points and boundary sampling points by \( \X = G^{}_{k,r} \subset \Omega \), \( \Y = \overline{G}^{}_{ k,r} \subset \pot \) and let \( \tilde{m} = \#(\X) \) and \( \bar{m} = \#(\Y) \) represent their respective cardinalities.
	Under the assumptions placed on the function spaces $\{\F, \G, \H\}$, the solution \( u \) to equation \eqref{1} is guaranteed to lie within a composite approximation class \( \mathcal{U} \). The best possible rate at which $u$ can be recovered from data is then governed by the expression $\max \left\{ \tilde m^{-s/d}, \, \bar m^{-(\bar{s} - 1)/(d-1)}, \,\tilde m^{-(\check{s} - 1)/d}\right\}$. For any function $(v,\mu)$, we introduce the norm $\|(v,\mu)\|_{\mathcal{U}} := \max \left\{ \|\Delta v+\mu\|_{\BB}, \, \|\operatorname{Tr}(v)\|_{\operatorname{Tr}(\bar \BB)}, \, \|v\|_{\check\BB},\, \|\mu\|_{\BB} \right\}$. The following theorem provides an upper bound for the error \( \|u - v\|_{H^1(\Omega)} \) based on the discrete loss functional \( \L^{c}_M(v) \). This result holds under the regularity assumptions that \( \Delta v+ \mu \in \BB\) with \(v \in \overline\BB \), and \( v \in \check\BB \).
	
	\begin{theorem}
		Let \( (u, \lambda) \) be the exact solution of the obstacle problem \eqref{prob2}, where \( f \in \mathcal{F} = \mathcal{U}(\BB) \), \( g \in \mathcal{G} \), and obstacle \( \psi \in \mathcal{H} \), as described in \eqref{3.24}. Assume that observations are available at discrete interior and boundary sites \( G^{}_{k,r} \) and \( \overline{G}_{ k,r} \) with corresponding sample sizes \( \tilde{m} \) and \( \bar{m}\). Then, for any admissible pair \( (v,\mu) \in K^s(\Omega) \), the discrete mixed residual functional \( \L^c_M\) defined in \eqref{discreteL} satisfies 
		\begin{align}\label{7.3}
			\|(u,\lambda) - (v,\mu)\|^2_{U} \lesssim \L^c_M(v,\mu) + \left( \|(v,\mu)\|^2_{\mathcal{U}} + 1\right)\mathfrak{R}_{\mathcal{U}}(\tilde{m}, \bar{m}),
		\end{align}
		where the constant in the inequality does not depend on \( u, v, \tilde{m}\), and \(\bar{m} \) and the residual term is given by
		\begin{align}\label{7.4}
			\mathfrak{R}_{\mathcal{U}}(\tilde{m}, \bar{m}) :=
			\max \left\{ \tilde{m}^{-\frac{2s}{d}},\, \bar{m}^{-\frac{2(\bar{s} - 1)}{d - 1}},\, \tilde{m}^{- \frac{2(\check s-1)}{d}} \right\}.
		\end{align}
	\end{theorem}
	
	\begin{proof}
		From Theorem~\ref{thm_2.1}, the obstacle problem admits the mixed stability estimate. For all $(v,\mu)\in K^s$, it holds
		\begin{align}\label{5.10}
			\|(u,\lambda)-(v,\mu)\|^2_{U}\lesssim J(v,\mu),
		\end{align}
		where the mixed functional $J$ is defined by~\eqref{J(u)}. Hence, it suffices to bound $J(v,\mu)$ in terms of the discrete loss. By construction of the discrete interior norm and the Sobolev embedding $L^\gamma(\Omega)\hookrightarrow H^{-1}(\Omega)$ with $\gamma =2d/(d+2)$, we have the norm equivalence 
		\begin{align*}
			\|\Delta v +\mu + f\|_{-1}&\leq 	\|\Delta v +\mu + f\|_{L^\gamma(\Omega)} \\
			&\leq \|\Delta v +\mu + f\|^*_{L^\gamma(\Omega)} + \|\Delta v +\mu + f\|_{\BB}\tilde{m}^{-\frac{s}{d}}.
		\end{align*}
		Similarly, the norm corresponding to the boundary term satisfies the estimate
		\begin{align*}
			\| \operatorname{Tr}(v) - g\|_{H^{1/2}(\partial\Omega)}\leq	\| \operatorname{Tr}(v) - g\|^*_{H^{1/2}(\partial\Omega)}\ + \|\operatorname{Tr}(v)-g\|_{\operatorname{Tr}(\bar \BB)} \bar{m}^{-\frac{s-1}{d-1}} .
		\end{align*}
			The last term in $J(v,\mu)$ is the duality pairing $\langle \mu, v-\psi\rangle$. From Section~\ref{sec_5.3}, the continuous and discrete duality pairings satisfy
			\begin{align*}
				|\langle \mu, v-\psi\rangle|	\le	|\langle \mu, v-\psi\rangle^*| + \left( \|\mu\|_{{\BB}}^2+\|v-\psi\|^2_{\check{\BB}}\right) \tilde m^{-\min\{\frac{s}{d},\frac{\check s-1}{d}\}}.
			\end{align*}
			Collecting the above estimates and combining them with the stability bound \eqref{5.10}, we obtain
			\begin{align*}
				\|(u,\lambda)-(v,\mu)\|^2_{U}&\leq  \|\Delta v + \mu+f\|^2_{-1} + \|v-g\|^2_{H^{1/2}(\partial\Omega)} + \langle\mu, v-\psi\rangle\\
				&\leq  \|\Delta v + \mu+f\|^{2,*}_{L^{\gamma}(\Omega)} + \|v-g\|^{2,*}_{H^{1/2}(\partial\Omega)} + \langle\mu, v-\psi\rangle^* \\&+ \|\Delta v +\mu + f\|^2_{\BB}\tilde{m}^{-\frac{2s}{d}} + \|\operatorname{Tr}(v)-g\|^2_{\operatorname{Tr}(\bar \BB)} \bar{m}^{-\frac{2(\bar s-1)}{d-1}} \\&+ \left( \|\mu\|_{{\BB}}^2+\|v-\psi\|^2_{\check{\BB}}\right) \tilde m^{-2\min\{\frac{s}{d}, \frac{\check s-1}{d}\}}\\
				&\lesssim  \L^c_M(v,\mu) + \left( \|(v,\mu)\|^2_{\mathcal{U}} + 1\right)\mathfrak{R}_{\mathcal{U}}(\tilde{m}, \bar{m}),
			\end{align*}
			where the second estimate follows directly from the continuous embedding $L^\gamma(\Omega)\hookrightarrow H^{-1}(\Omega)$ together with the equivalence between the discrete and continuous norms. The argument is completed by using the bound $\|f\|_{\BB}\leq 1$, $\|g\|_{\operatorname{Tr}(\bar{\BB})}\leq 1$ and $\|\psi\|_{\check{\BB}}\leq 1$, which concludes the proof. 
		\end{proof}

		\section{Numerical results}\label{Sec_6}
		In this section, we demonstrate the application of the introduced PINNs framework to a series of numerical examples. For all case studies, we employ fully connected neural networks to approximate the latent functions that represent both the PDE solutions and the unknown boundary conditions. In the numerical experiments, we focus on the mixed PINNs and CPINNs loss, which provides a numerically robust and implementable realization of the primal stability-based loss introduced in the theoretical analysis.
		
		The numerical experiments are organized as follows. We first consider smooth solutions with homogeneous boundary conditions to verify basic convergence properties. We then examine smooth solutions with non-homogeneous boundary data and increasingly complex obstacle geometries, leading to nontrivial contact sets. Next, we demonstrate the applicability of the framework to vector-valued and higher-order obstacle problems, including linear elasticity, and biharmonic models. Finally, in the last experiment, we consider an obstacle problem with a non-smooth exact solution in order to assess the performance of the method under reduced regularity.
		
		\subsection*{Computational setup}
		For all the experiments, we use the following architectural and computational setup to ensure stable training and accurate approximation of the obstacle problem.
		\begin{itemize}
			\item \textbf{Network architecture: }For computation of all the experiments, we use a fully connected neural network of depth 3 with 15 neurons on each hidden layer.
			\item \textbf{Activation function: }The $\mathrm{ReLU}^3$ activation function is used in the hidden layers, while $\tanh$ is used in the output layer.
			\item \textbf{Parameter initialization: }Weights are initialized using Xavier and He initialization~\cite{glorot2010understanding} to maintain stable variance across layers.
			\item \textbf{Optimization scheme: }The optimization is performed using the energy natural gradient descent scheme for updating the network weights, with training carried out for 500 epochs.
		\end{itemize}
		\subsection{Experiment 1}\label{exp1}
		As a first numerical test, we consider the benchmark problem with the manufactured solution 
		\begin{align*}
			u_1(x,y) =  \sin{\pi x}\sin{\pi y},\ \quad (x,y) \in (0,1)^2,	
		\end{align*}
		with the obstacle 
		\begin{align*}
			\psi_1(x,y) = 
			\begin{cases}
				\sin{\pi x}\sin{\pi y}, & x\leq 1/2, \\
				(128x^3-240x^2+144x-27)\sin{\pi y}, & x\in(1/2,3/4),\\
				0,&x\geq 3/4.
			\end{cases}
		\end{align*}
		This solution satisfies homogeneous Dirichlet boundary conditions and the forcing term is taken as 
		\begin{align*}
			f(x,y) = 
			\begin{cases}
				0, &x<1/2,\\
				-\Delta u, &x\geq 1/2.
			\end{cases}
		\end{align*} 
		The solution is smooth, vanishes on the boundary, and is commonly used in the PDE literature as a prototype test case. Substituting the solution into the governing PDE operator yields the corresponding boundary data \(g\).
		The corresponding error measurements in the relative $H^1(\Omega)$ norm and graphical illustrations are provided below. 
		
		\begin{table}
			\centering
			\renewcommand{\arraystretch}{1.2}
			\begin{tabular}{|c|c c c c|c|}
				\hline
				Grid &\multicolumn{4}{c|}{$u_{1}(x,y) $} \\
				\cline{2-5}
				points & Error $(\%)$& Loss $(10^{-6})$  & $\mathfrak{R}_{\mathcal{U}}(\tilde{m}, \bar{m})$ & $ \|(v,\mu)\|_{\U}^2$ \\
				$N$ & PINNs / CPINNs & $\L^p_{M}$ / $\L_{M}^c$ & 	  & PINNs / CPINNs  \\
				\hline
				$5$ & 1.25 / 1.41 & $1.0$ / $0.0$ & $3.90 \times 10^{-3}$ & $19.04$ / $19.02$ \\
				$10$ &  1.01 / 0.42 & $0.41$ / $8.0$ &  $7.71\times 10^{-4}$  & $17.90$ / $17.91$ \\
				$15$ & 0.59 / 0.53 & $7.0$ / $1.9$ & $3.18 \times 10^{-4}$ & $18.21$ / $18.21$ \\
				$20$ & 0.55/ 0.51 & $6.0$ / $1.7$ & $1.73 \times 10^{-4}$  & $21.21$ / $21.10$ \\
				$25$ & 0.57 / 0.53 & $2.3$ / $1.4$ &  $1.08 \times 10^{-4}$  & $17.96$ / $17.90$ \\
				$30$ & 0.43 / 0.32 & $2.2$ / $0.3$& $7.43 \times 10^{-5}$ & $18.14$ / $18.16$ \\
				\hline
			\end{tabular}
			\caption{Error comparison of solution \(u_1\) obtained by using CPINNs and standard PINNs.}
			\label{table_11}
		\end{table}
		The numerical results reported in Table~\ref{table_11} are consistent with the stability estimate
		\begin{equation}\label{eq:exp1_stability}
			\|(u,\lambda)-(v,\mu)\|_{U}^{2}
			\;\lesssim\;
			\mathcal{L}_{M}^{c}(v,\mu)
			+
			\big(\|(v,\mu)\|_{\U}^{2}+1\big)\,
			\mathfrak{R}_{\mathcal{U}}(\tilde m,\bar m),
		\end{equation}
		where the recovery residual $\mathfrak{R}_{\mathcal{U}}(\tilde m,\bar m)$ is evaluated for the
		Besov smoothness index $s=2$. As the number of interior and boundary collocation points increases, both the discrete mixed loss $\mathcal{L}_{M}^{c}$ and the corresponding recovery residual decrease. This decay is reflected in the numerical data by a systematic reduction of the $\mathcal{U}$-norm error, which confirms the consistency of the proposed discretization.

		\subsection{Experiment 2}
		As a second numerical experiment, we examine a smooth oscillatory solution with non-homogeneous boundary condition given by
		\begin{align*}
			u_2(x,y) = 0.6 - (x-0.5)^2 - (y-0.5)^2 + 0.1\sin(\pi x) \sin(\pi y),\ \quad (x,y) \in (0,1)^2,	
		\end{align*}
		with the obstacle 
		\begin{align*}
			\psi_2(x,y) = 0.6 - (x-0.5)^2 - (y-0.5)^2.
		\end{align*}
		The forcing term is defined in the same way as in Experiment~\ref{exp1}. The corresponding error results and plots for both experiments are presented below. In addition, Table~\ref{table_1} reports the verification of the stability estimate,
		which demonstrates the consistency between the discrete loss and the observed error behavior.
		\begin{table}
			\centering
			\renewcommand{\arraystretch}{1.1}
			\begin{tabular}{|c|c c c c|c|}
				\hline
				Grid &\multicolumn{4}{c|}{$u_{2}(x,y) $} \\
				\cline{2-5}
				points & Error $(\%)$& Loss $(10^{-6})$  & $\mathfrak{R}_{\mathcal{U}}(\tilde{m}, \bar{m})$ & $ \|(v,\mu)\|_{\U}^2$ \\
				$N$ & PINNs / CPINNs & $\L^p_{M}$ / $\L_{M}^c$ & 	  & PINNs / CPINNs  \\
				\hline
				$5$ &  0.72 / 0.68 & $0.11$ / $0.02$ &  $3.90 \times 10^{-3}$ & $21.22$ / $21.23$ \\
				$10$ &  0.24 / 0.17 & $0.09$ / $0.02$ &  $7.71\times 10^{-4}$  & $22.89$ / $22.87$ \\
				$15$ & 0.24 / 0.15 & $0.09$ / $0.01$ & $3.18 \times 10^{-4}$ & $23.33$ / $23.32$ \\
				$20$ & 0.24/ 0.14 & $0.07$ / $0.01$ & $1.73 \times 10^{-4}$  & $22.80$ / $22.80$ \\
				$25$ & 0.26 / 0.15 & $0.09$ / $0.01$ &  $1.08 \times 10^{-4}$  & $23.59$ / $23.54$ \\
				$30$ & 0.23 / 0.13& $0.08$ / $0.01$& $7.43 \times 10^{-5}$ & $23.21$ / $23.19$ \\
				\hline
			\end{tabular}
			\caption{Error comparison of solution \(u_2\) obtained by using CPINNs and standard PINNs.}
			\label{table_1}
		\end{table}
		
		\begin{figure}
				\subfloat{
					{\includegraphics[width=0.32\textwidth]{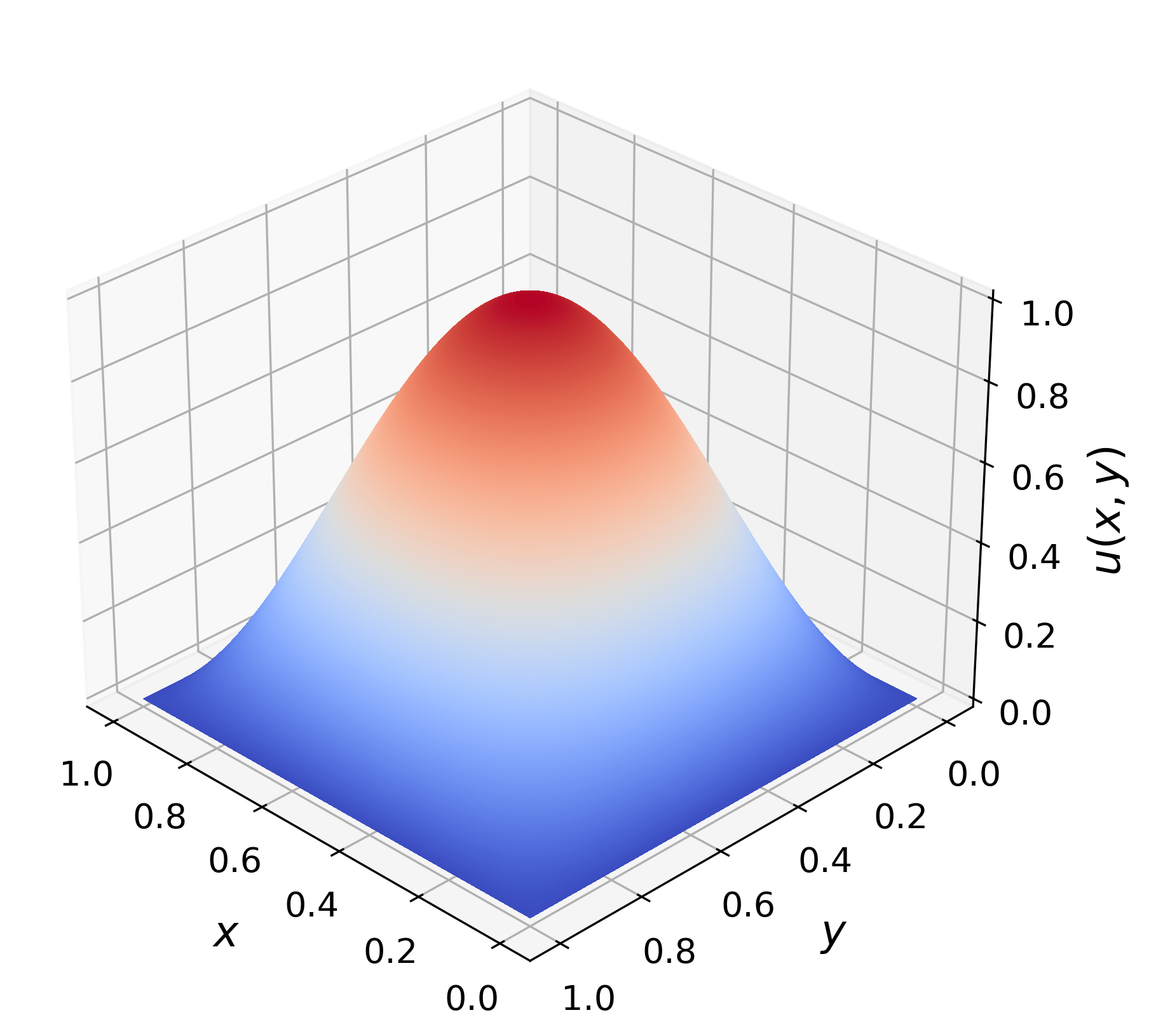}} 
					{\includegraphics[width=0.32\textwidth]{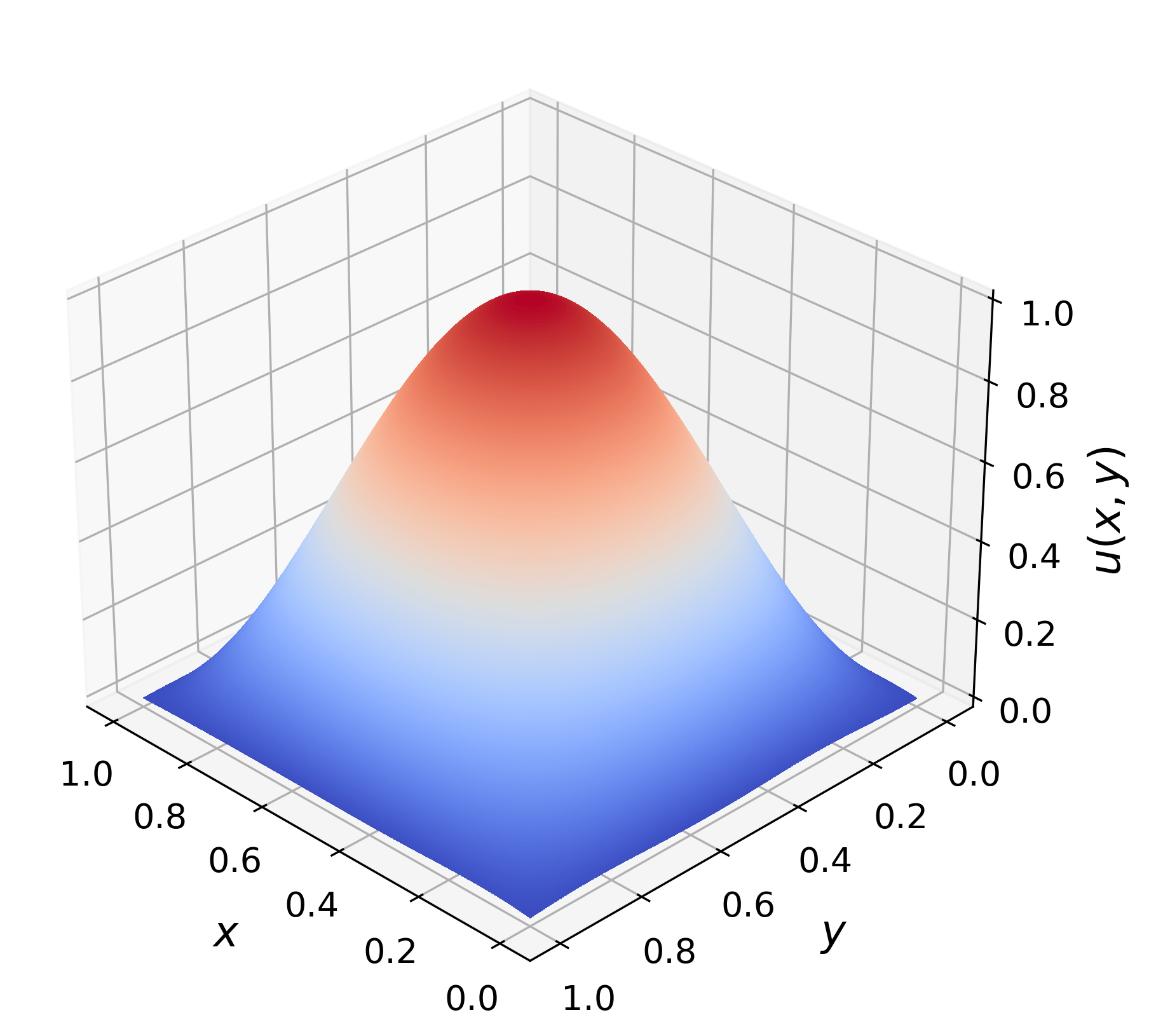}}
					{\includegraphics[width=0.32\textwidth]{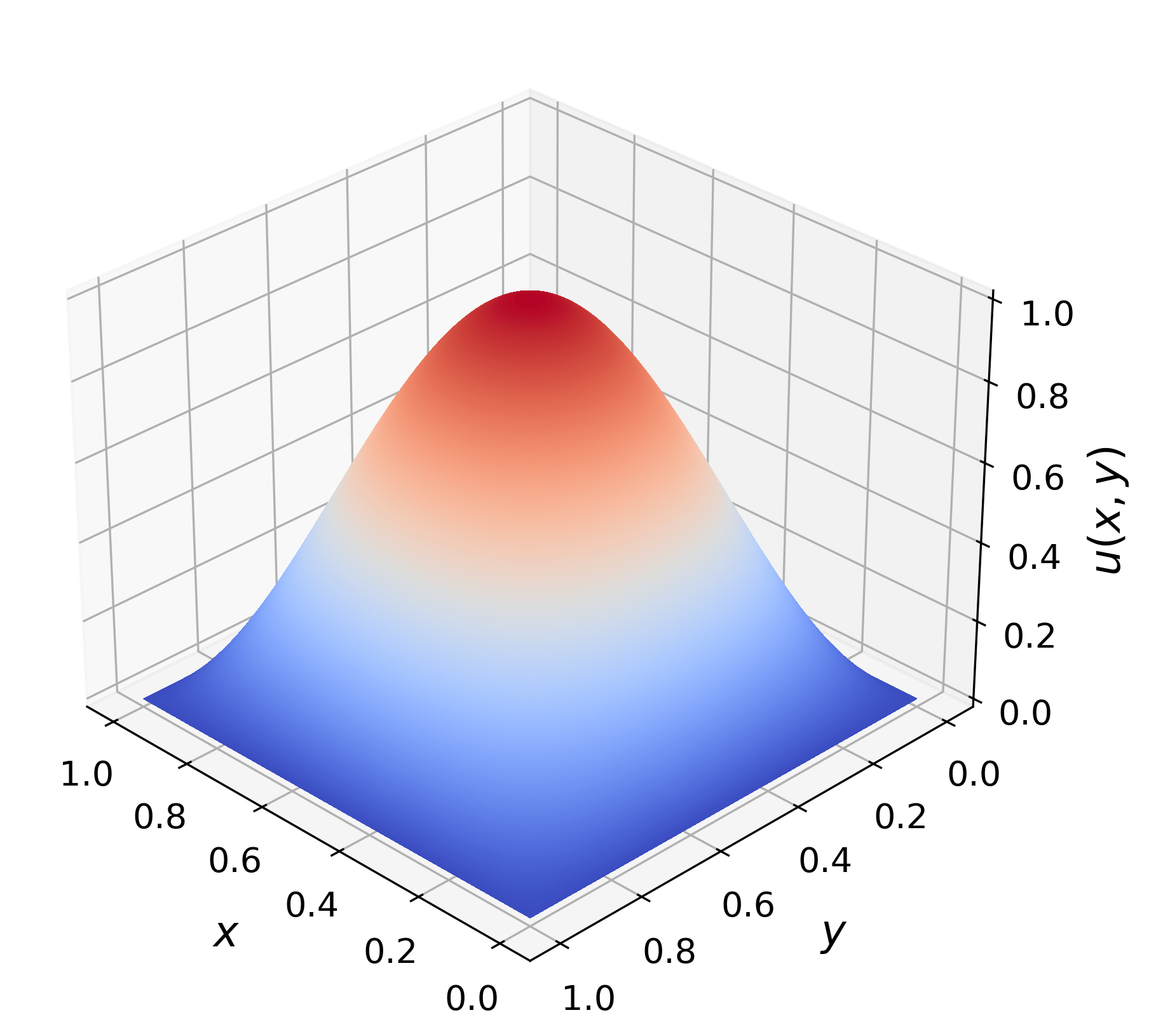}}
				}
				
				\subfloat{
					{\includegraphics[width=0.32\textwidth]{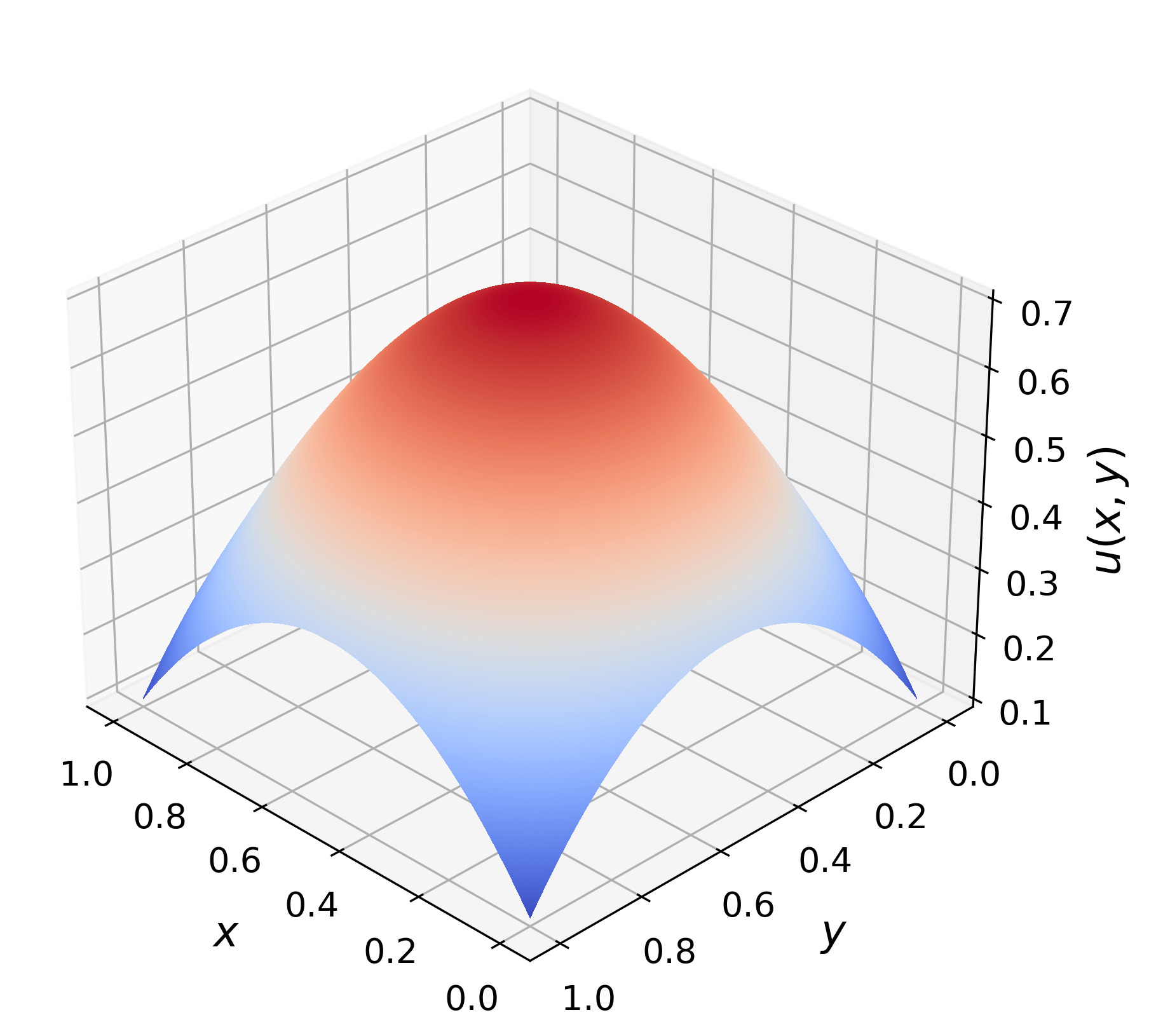}} 
					{\includegraphics[width=0.32\textwidth]{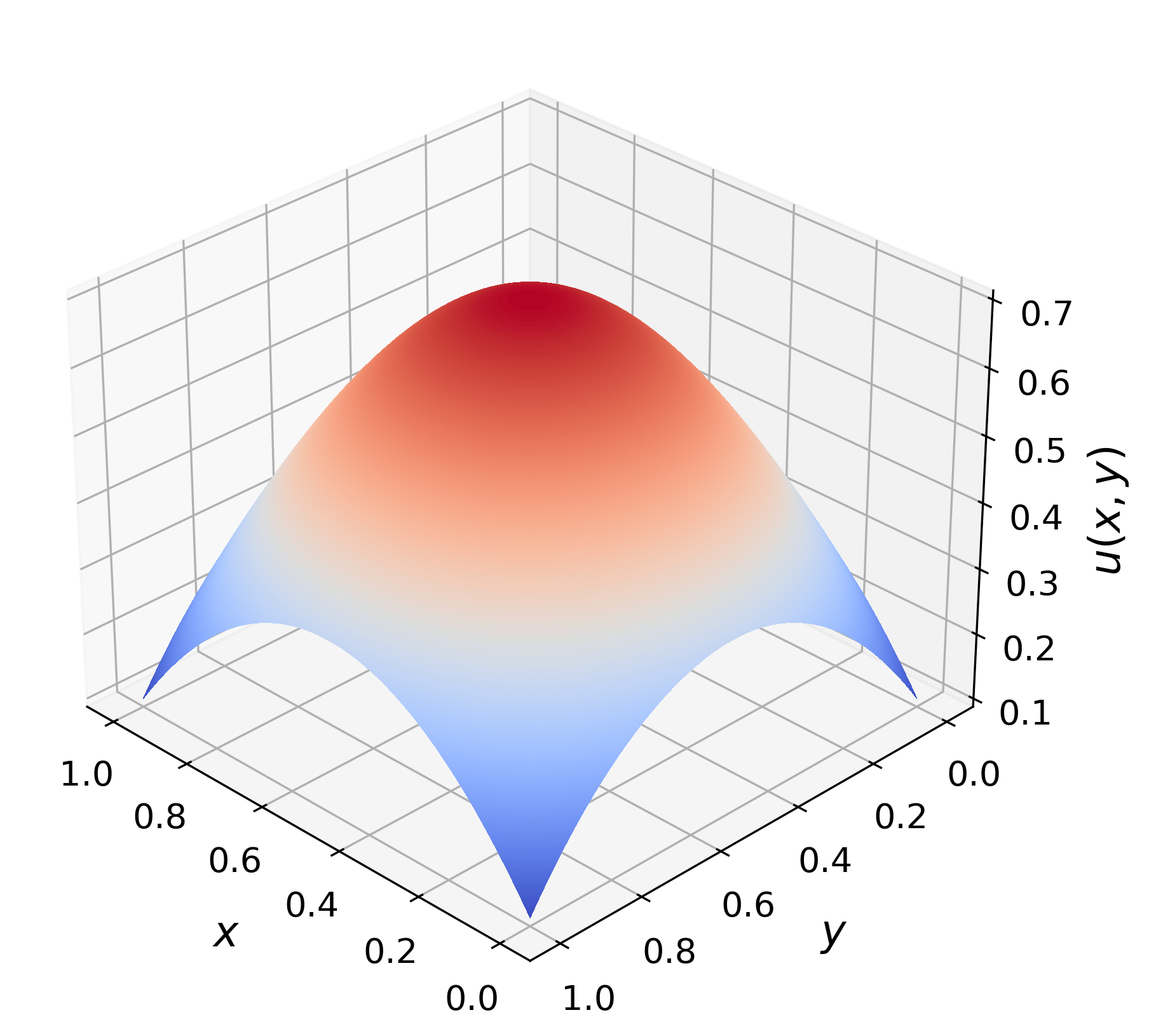}}
					{\includegraphics[width=0.32\textwidth]{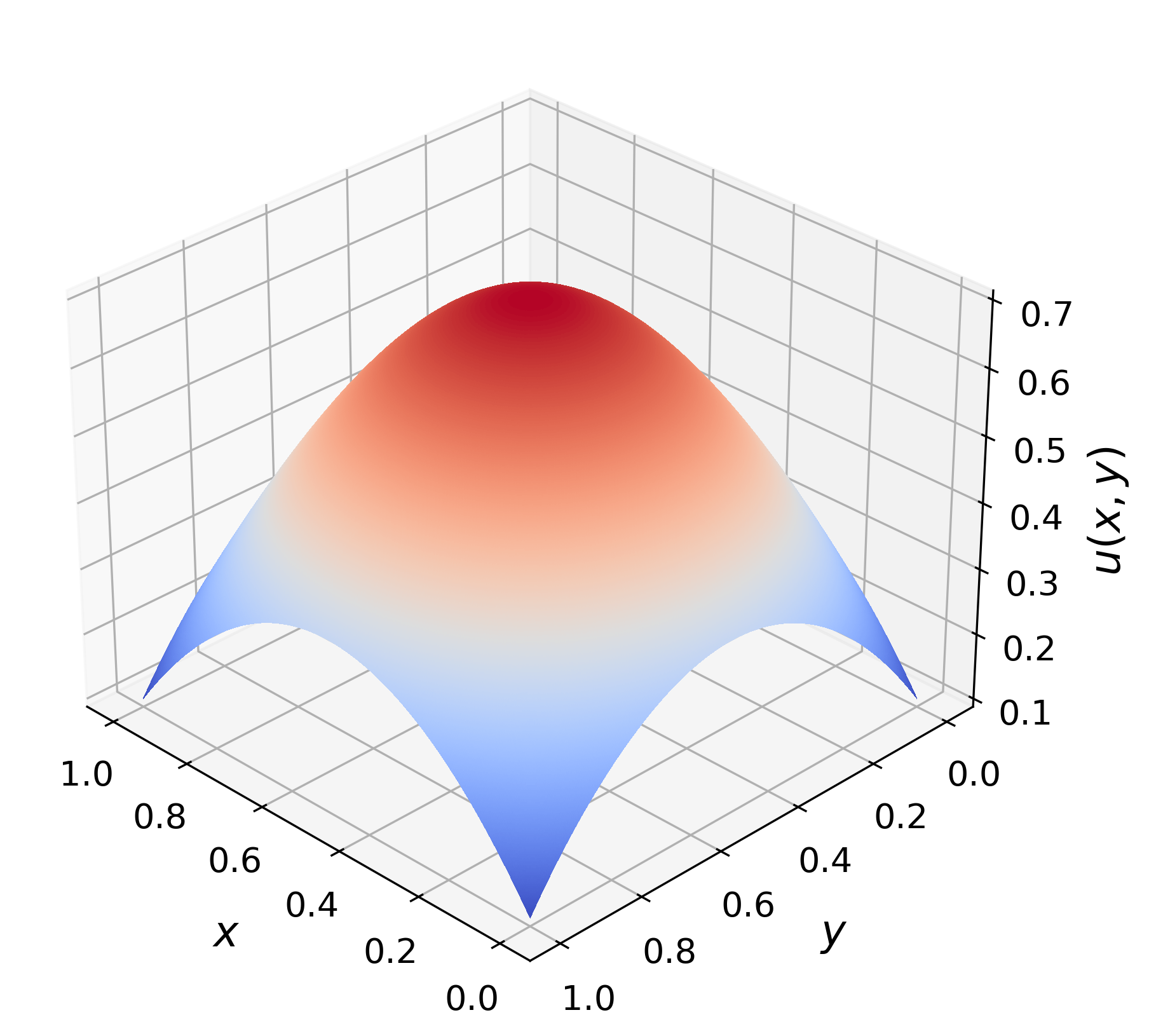}}
				}
				
				\caption{Visualization of the true solution (left), standard PINNs prediction using $\L^p_{M}$ (middle) and CPINNs prediction using $\L_{M}^{c}$ (right), evaluated on $15 \times 15$ grid for $u_1$ (top row) and $u_2$ (bottom row).}
		\end{figure}

\subsection{Experiment 3}
In this experiment, we consider a non-homogeneous boundary value problem with the exact solution
\begin{align*}
u_3(x,y) = x+y+xy(1-x)(1-y),\ \quad (x,y) \in (0,1)^2,	
\end{align*}
and the obstacle function is chosen as
\begin{align*}
\psi_3(x,y) =0.2(x+y)^2,
\end{align*}
which lies below the exact solution in most of the domain and induces a nontrivial contact region. The forcing term is again defined as in Experiment~\ref{exp1}. The numerical errors and the corresponding graphical representations for this experiment are presented below.

\subsection{Experiment 4}
As another non-homogeneous obstacle problem, we consider the exact solution
\begin{align*}
u_4(x,y) = \sin(\pi x) + \sin(\pi y)+0.1 xy(1-x)(1-y),\ \quad (x,y) \in (0,1)^2,	
\end{align*}
and the obstacle is designed in a ring-shaped form and is given by
\begin{align*}
\psi_4(x,y) =0.8-4[(x-0.5)^2+(y-0.5)^2-0.1]^2.
\end{align*}
This obstacle generates an annular active region around the center of the domain, while the solution remains free inside the inner ring and near the boundary. The corresponding forcing term is given by $f(x,y) = 	8((x-0.5)^2+(y-0.5)^2)-1$ for region $0.1\leq(x-0.5)^2+(y-0.5)^2\leq 0.3$ and $f(x,y) = 	\pi^2( \sin(\pi x) + \sin(\pi y))+ 0.2(x+y-2xy)$, elsewhere. The numerical error results and the corresponding visualizations for this experiment are summarized below.
\begin{figure}

\subfloat{
{\includegraphics[width=0.32\textwidth]{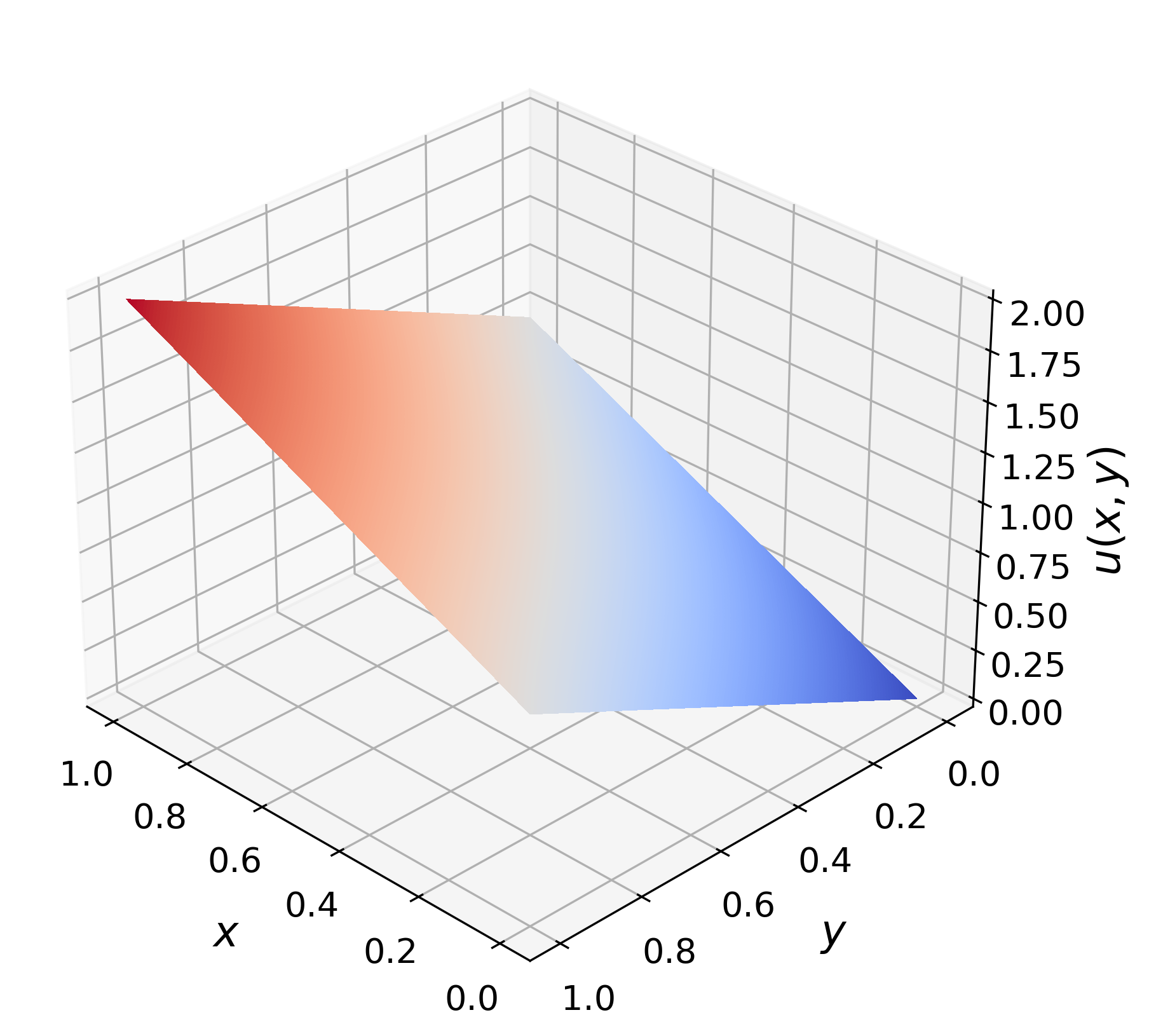}} 
{\includegraphics[width=0.32\textwidth]{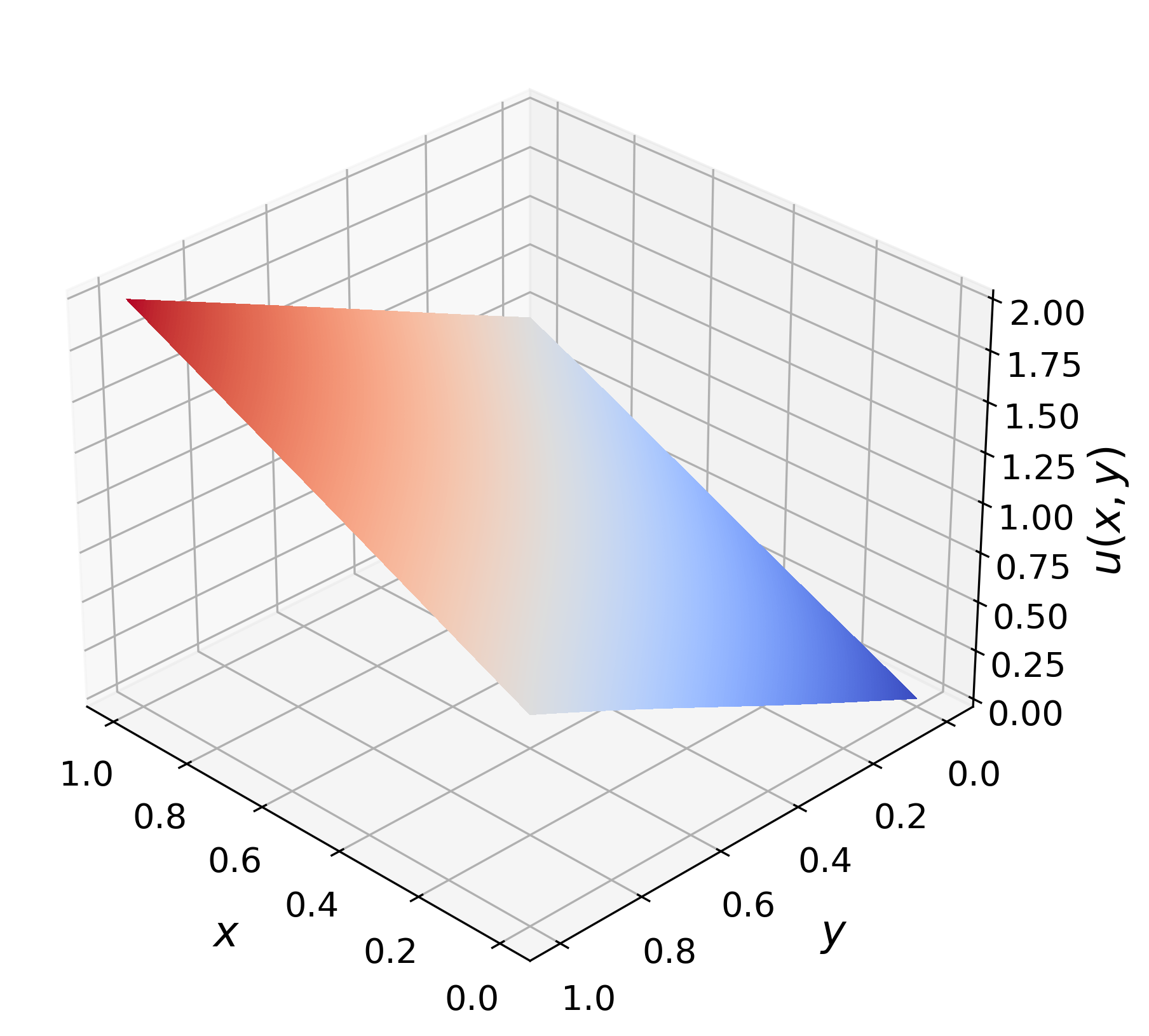}}
{\includegraphics[width=0.32\textwidth]{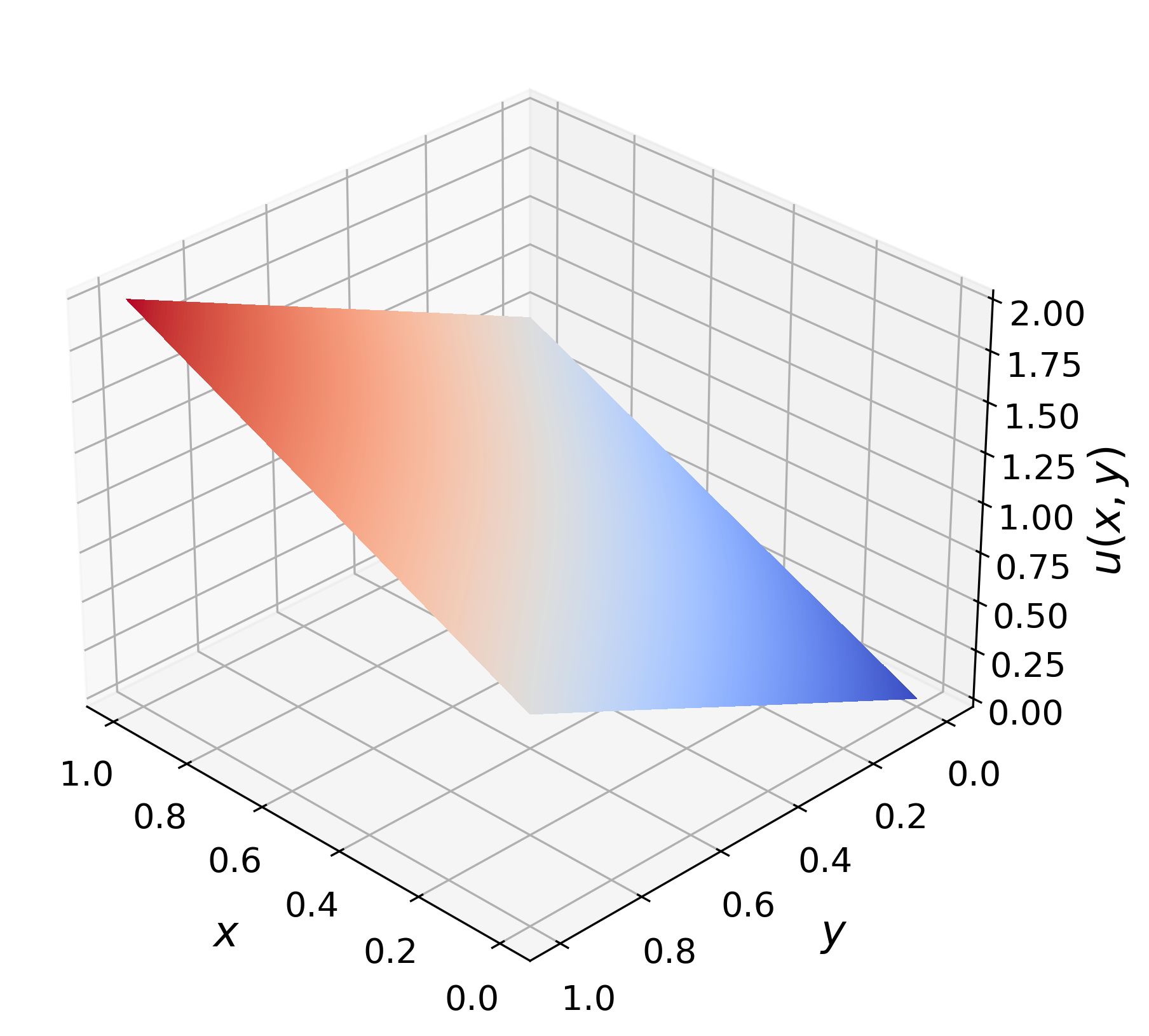}}
}

\subfloat{
{\includegraphics[width=0.32\textwidth]{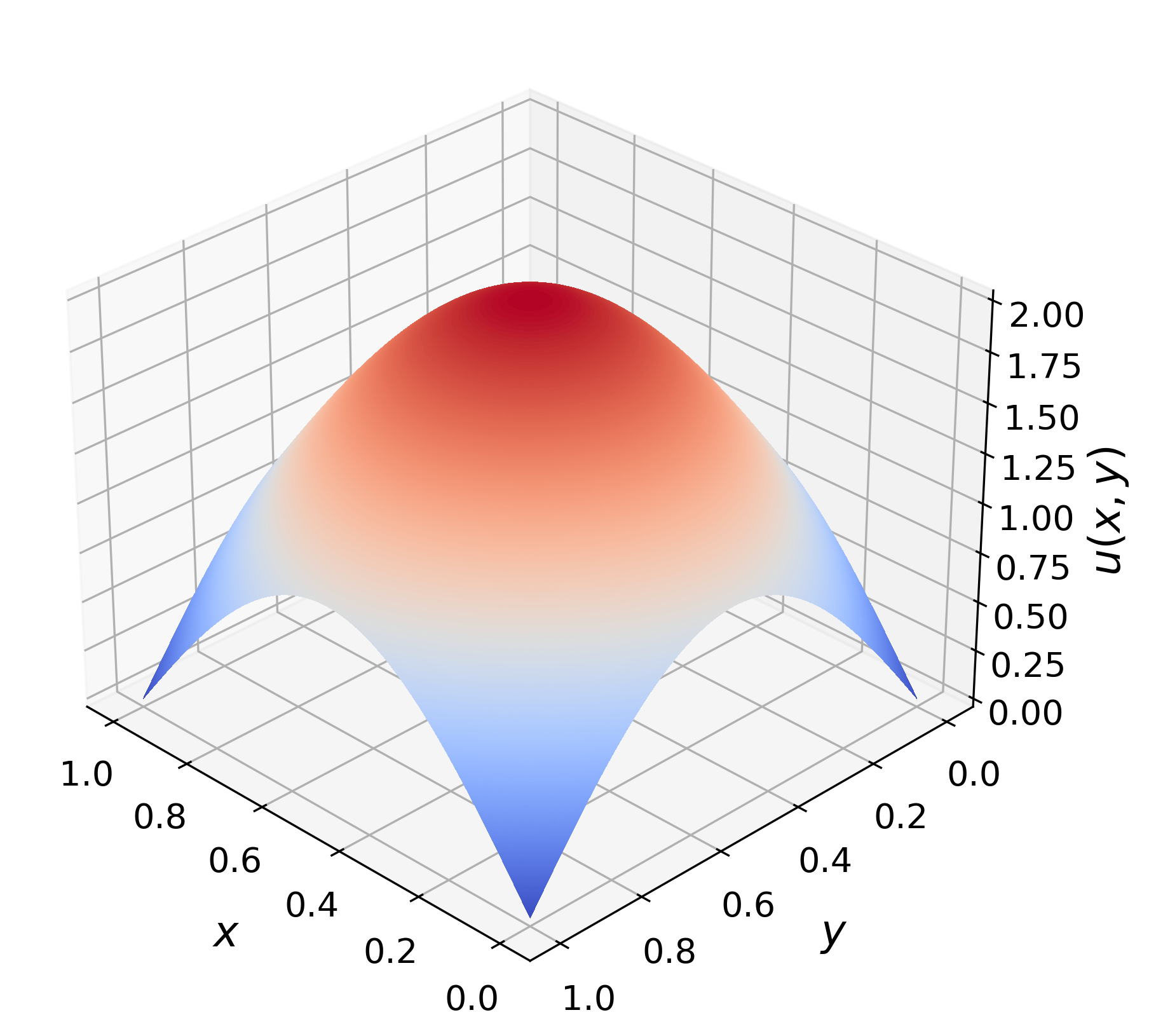}} 
{\includegraphics[width=0.32\textwidth]{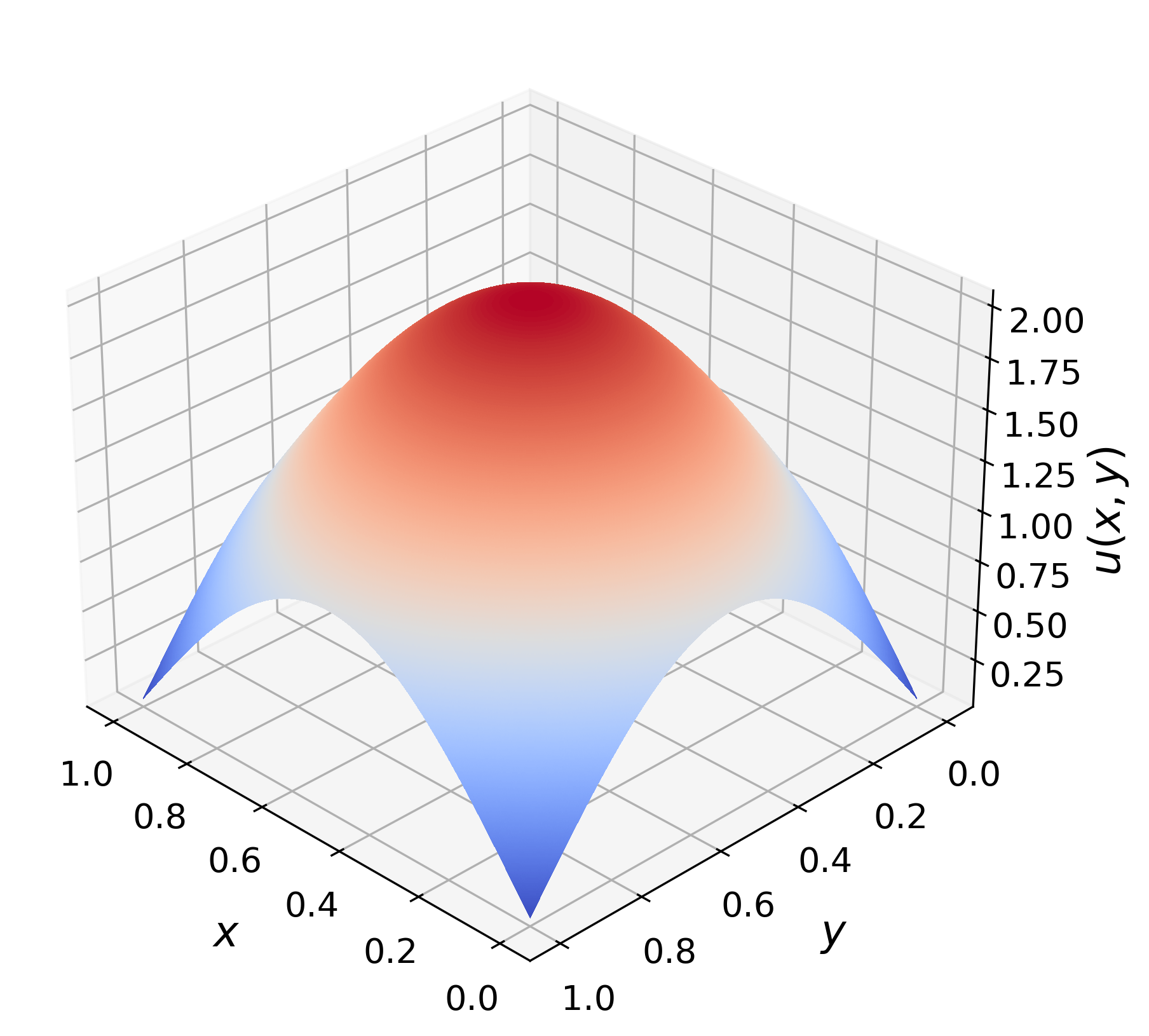}}
{\includegraphics[width=0.32\textwidth]{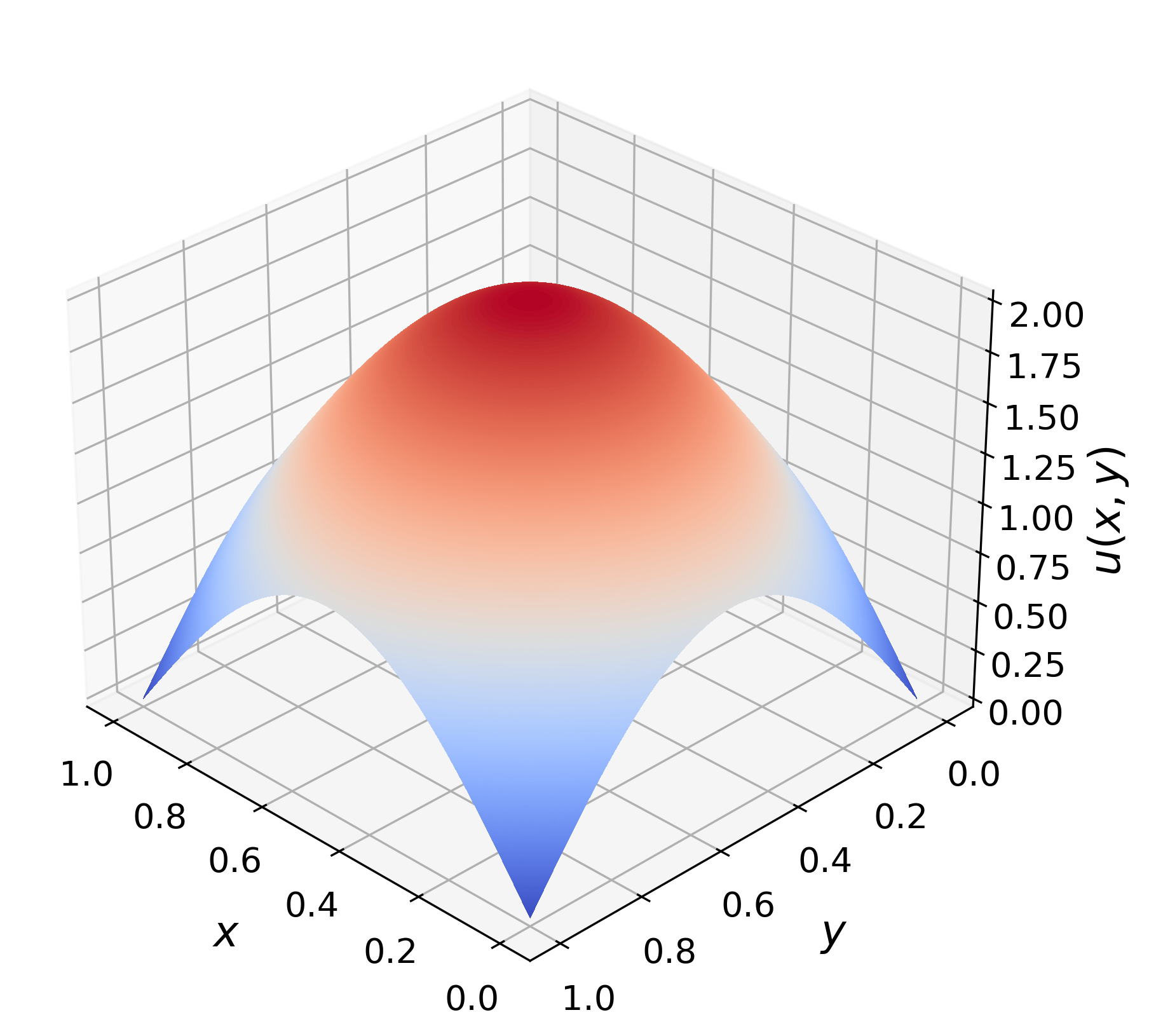}}
}

\caption{Visualization of the true solution (left), standard PINNs prediction using $\L^p_{M}$ (middle) and CPINNs prediction using $\L_{M}^{c}$ (right), evaluated on $15 \times 15$ grid for $u_3$ (top row) and $u_4$ (bottom row).}
\end{figure}

\begin{table}
\centering
\renewcommand{\arraystretch}{1.1}
\begin{tabular}{|c|c c|c c|c|}
\hline
Grid &\multicolumn{2}{|c|}{$u_3(x,y) $} &  \multicolumn{2}{c|}{$u_4(x,y)$} \\
\cline{2-5}
points & Error $(\%)$& Loss $(10^{-5})$ & Error$(\%)$ & Loss $(10^{-4})$ \\
$N$ & PINNs / CPINNs & $\L^p_{M}$ / $\L_{M}^c$ & PINNs / CPINNs & $\L^p_{M}$ / $\L_{M}^c$ \\
\hline
$5$ & 0.67 / 0.23 & $0.3$ / $0.1$ &  3.47 / 1.58 & $2.72$ / $1.88$ \\
$10$ & 0.56 / 0.13 & $0.7$ / $0.2$ &  2.60 / 0.93 & $1.36$ / $1.09$ \\
$15$ & 0.60 / 0.12 & $0.9$ / $0.3$ & 2.32 / 0.78 & $1.09$ / $0.90$ \\
$20$ & 0.59 / 0.10 & $0.9$ / $0.2$ & 2.12/ 0.68 & $0.96$ / $0.82$ \\
$25$ & 0.57 / 0.10 & $0.6$ / $0.1$ &  2.20 / 0.64 & $0.92$ / $0.91$ \\
$30$ & 0.56 / 0.09& $0.7$ / $0.1$ & 2.02 / 0.51& $0.84$ / $0.77$ \\
\hline
\end{tabular}
\caption{Error comparison of solutions \(u_3\) and \(u_4\) obtained by using CPINNs and standard PINNs.}
\label{table_3}
\end{table}

\subsection{Experiment 5 (Linear elasticity obstacle problem)} 
In this experiment, we consider a two dimensional linear elasticity model with a nonlinear smooth displacement field and a vector-valued obstacle that restricts the admissible deformation. The governing equilibrium equations are
\begin{align*}
\nabla\cdot\mathbb{T} \geq -\boldsymbol{F} \quad &\text{in } \Omega, \\
\boldsymbol{u} \geq \boldsymbol{\psi} \quad &\text{in } \Omega, \\
(\boldsymbol{u}-\boldsymbol{\psi})\,(-\nabla\cdot\mathbb{T} - \boldsymbol{F}) = 0 \quad &\text{in } \Omega,\\
\boldsymbol{u }= \boldsymbol{g }\quad &\text{on } \partial\Omega,
\end{align*}
where the stress tensor is $\mathbb{T} = \lambda \textbf{Tr}(\sigma)I+2\mu\sigma$ with Lam$\acute{e}$ parameters $\lambda, \mu$ and stress tensor are defined as
\begin{align}\label{strain}
\lambda = \frac{E\nu}{(1+\nu)(1-2\nu)},\quad \mu = \frac{E}{2(1+\nu)}, \ \text{and} \ \sigma(\boldsymbol{u}) = \frac{(\nabla \boldsymbol{u} + (\nabla\boldsymbol{ u})^T)}{2}.
\end{align}
We prescribe a smooth manufactured displacement vector field 
$$\boldsymbol{u_5} = ((1+x^2)(1+y^2)e^{(x+y)},\, (1+x^2)(1+y^2)e^{(x+y)}),$$ 
for $E =1 $ and $\nu = 0.3, \, 0.4,\, 0.49$ over the domain $\Omega = (0,1)^2$. We impose a radially symmetric obstacle centered at the midpoint of the domain
\begin{align*}
\boldsymbol{\psi_5} = (0.8-4[(x-0.5)^2+(y-0.5)^2-0.1]^2,\,0.8-4[(x-0.5)^2+(y-0.5)^2-0.1]^2), 
\end{align*}
and the body force is defined as
\begin{align*}
\boldsymbol F(x,y):=
\begin{cases}
\big(
5|\boldsymbol{r}|^2-0.5,\;
5|\boldsymbol{r}|^2-0.5
\big), &0.1 \le|\boldsymbol{r}|^2 \le 0.3,\\
-\nabla\cdot \mathbb T,& \text{otherwise},
\end{cases}
\end{align*}
for $|\boldsymbol{r}|^2= (x-0.5)^2+(y-0.5)^2$. We present below the obtained error values and the visualizations that express the behavior of the computed solution and the obstacle.
\begin{table}[H]
\centering
\footnotesize
\renewcommand{\arraystretch}{1.2}
\begin{tabular}{|c|c c|c c|c c|}
\hline
Grid & \multicolumn{2}{c|}{\boldsymbol{$E=1,\, \nu = 0.3$}} 
& \multicolumn{2}{c|}{\boldsymbol{$E=1,\, \nu = 0.4$}} 
& \multicolumn{2}{c|}{\boldsymbol{$E=1,\, \nu = 0.49$}} \\
\cline{2-7}
points & Error $(\%)$& Loss $(10^{-4})$ & Error $(\%)$ & Loss $(10^{-4})$ & Error $(\%)$ & Loss $(10^{-3})$ \\
$N$ & $\mathscr{P}$ / $\mathscr{C}$ & $\L^p_{M}$ / $\L_{M}^c$ & $\mathscr{P}$ / $\mathscr{C}$& $\L^p_{M}$ / $\L_{M}^c$ & $\mathscr{P}$ / $\mathscr{C}$ & $\L^p_{M}$ / $\L_{M}^c$ \\
\hline
$5$ & 0.24 / 0.05 & $0.69$ / $0.34$ &  0.18 / 0.20 & $0.15$ / $1.61$ & 3.32 / 3.54 & $1.50$ / $6.97$ \\
$10$ & 0.19 / 0.04 & $0.29$ / $0.12$ &  0.35/ 0.10 & $4.07$ / $2.29$ & 1.47 / 1.18 & $2.01$ / $2.19$ \\
$15$ & 0.55 / 0.016 & $2.88$ / $0.89$ & 0.14 / 0.057 & $1.12$ / $3.48$ & 2.60 / 1.02 & $1.37$ / $0.84$ \\
$20$ & 0.15 / 0.016 & $1.52$ / $0.37$ &  0.17/ 0.025 & $0.82$ / $1.64$ & 2.54/ 0.69 & $1.31$ / $0.59$ \\
$25$ & 0.14 / 0.013 & $1.16$ / $0.42$ &  0.16/ 0.022 & $1.83$ / $1.53$ & 2.20 / 0.54 & $2.12$ / $1.43$ \\
$30$ & 0.14 / 0.012 & $0.84$ / $0.23$ & 0.12 / 0.014 & $1.65$ / $1.88$ & 2.32 / 0.58 & $1.77$ / $1.45$ \\
\hline
\end{tabular}
\caption{Error comparison of solutions obtained by using CPINNs($\mathscr{C}$) and standard PINNs($\mathscr{P}$) for $\boldsymbol{u_5}$ on $\nu = 0.3,\, 0.4$ and $0.49$.}
\label{table_4}
\end{table}

\begin{figure}
\centering
\subfloat{
{\includegraphics[width=0.32\textwidth]{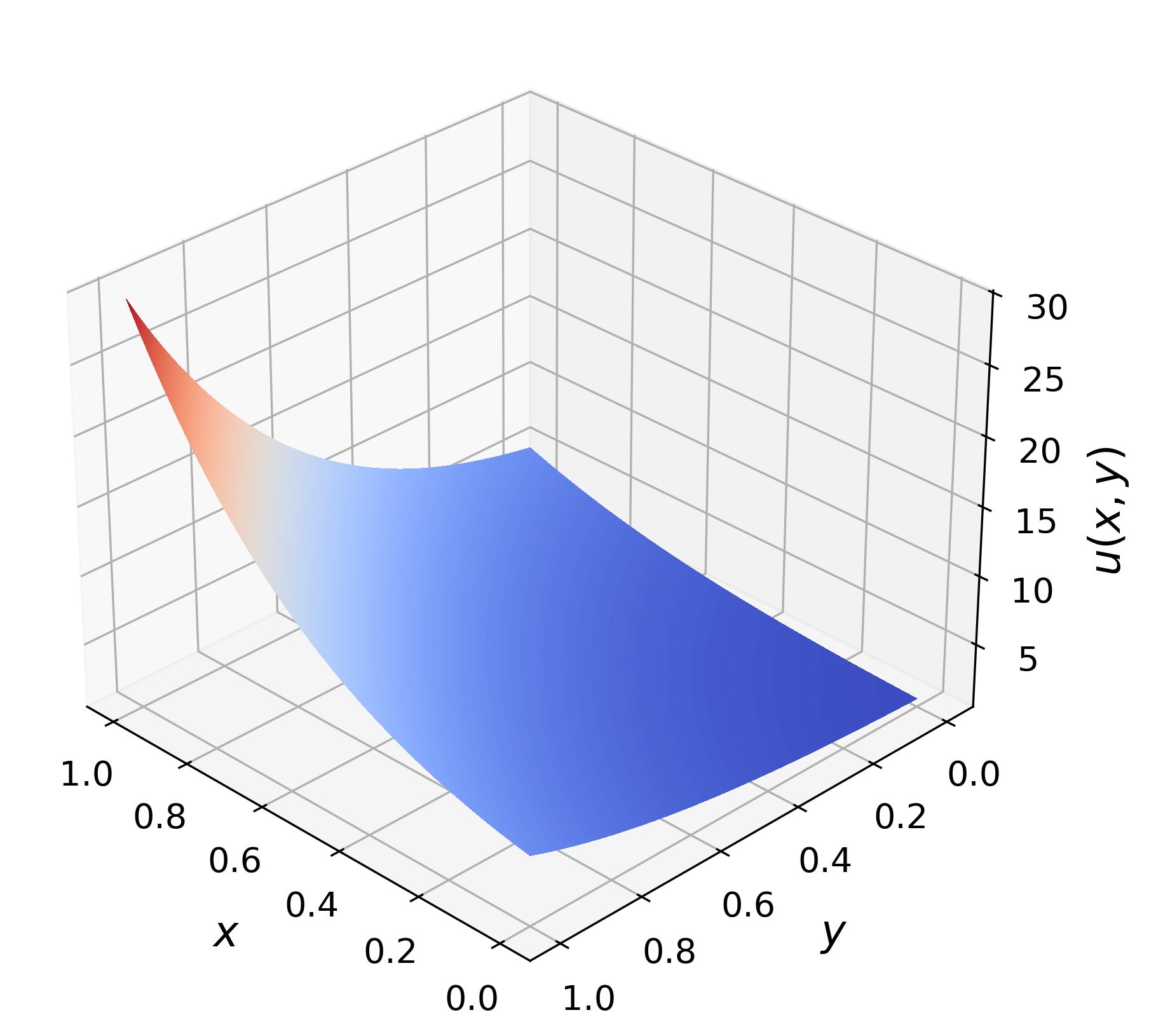}} 
{\includegraphics[width=0.32\textwidth]{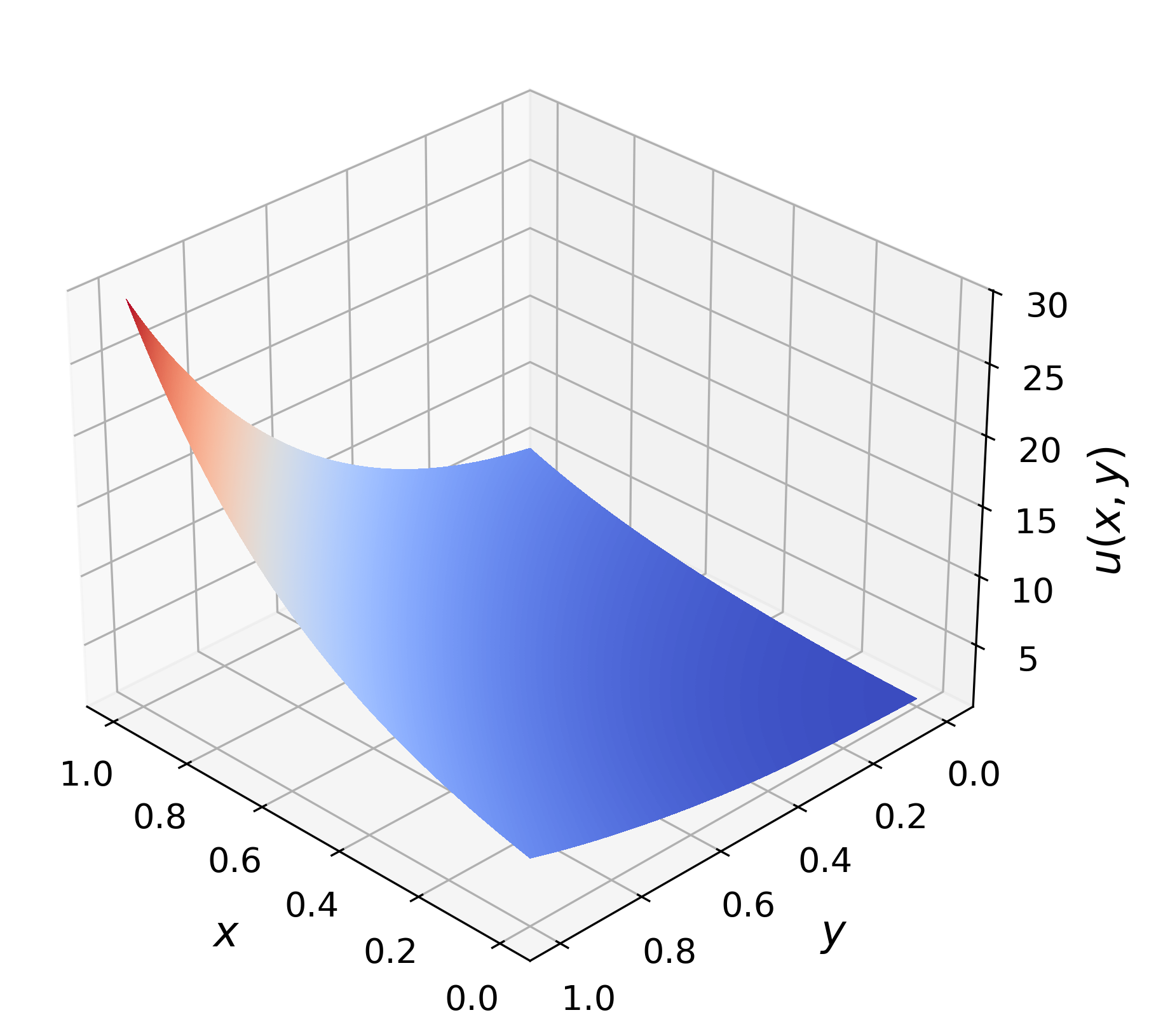}}
{\includegraphics[width=0.32\textwidth]{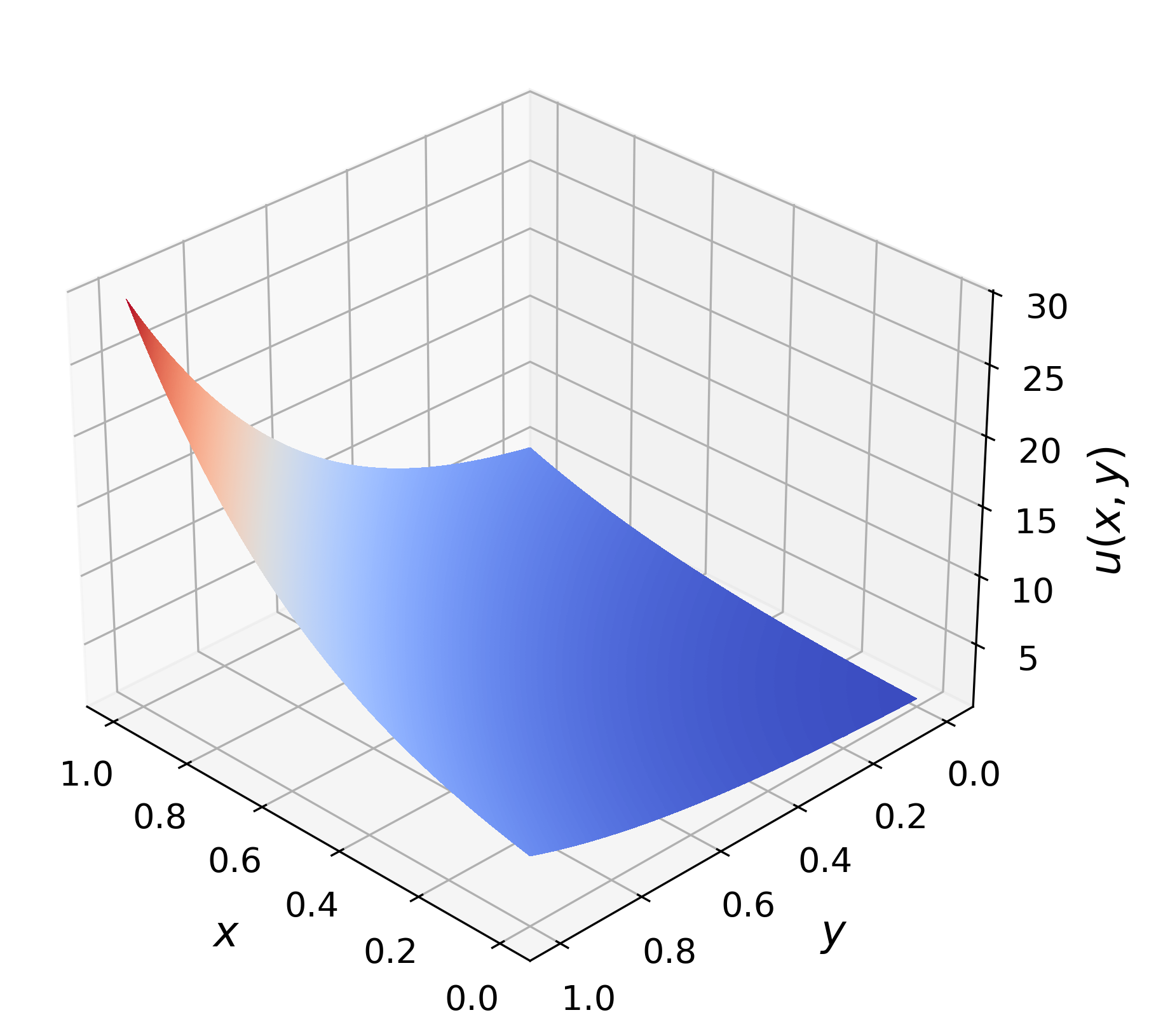}}
}

\subfloat{
{\includegraphics[width=0.32\textwidth]{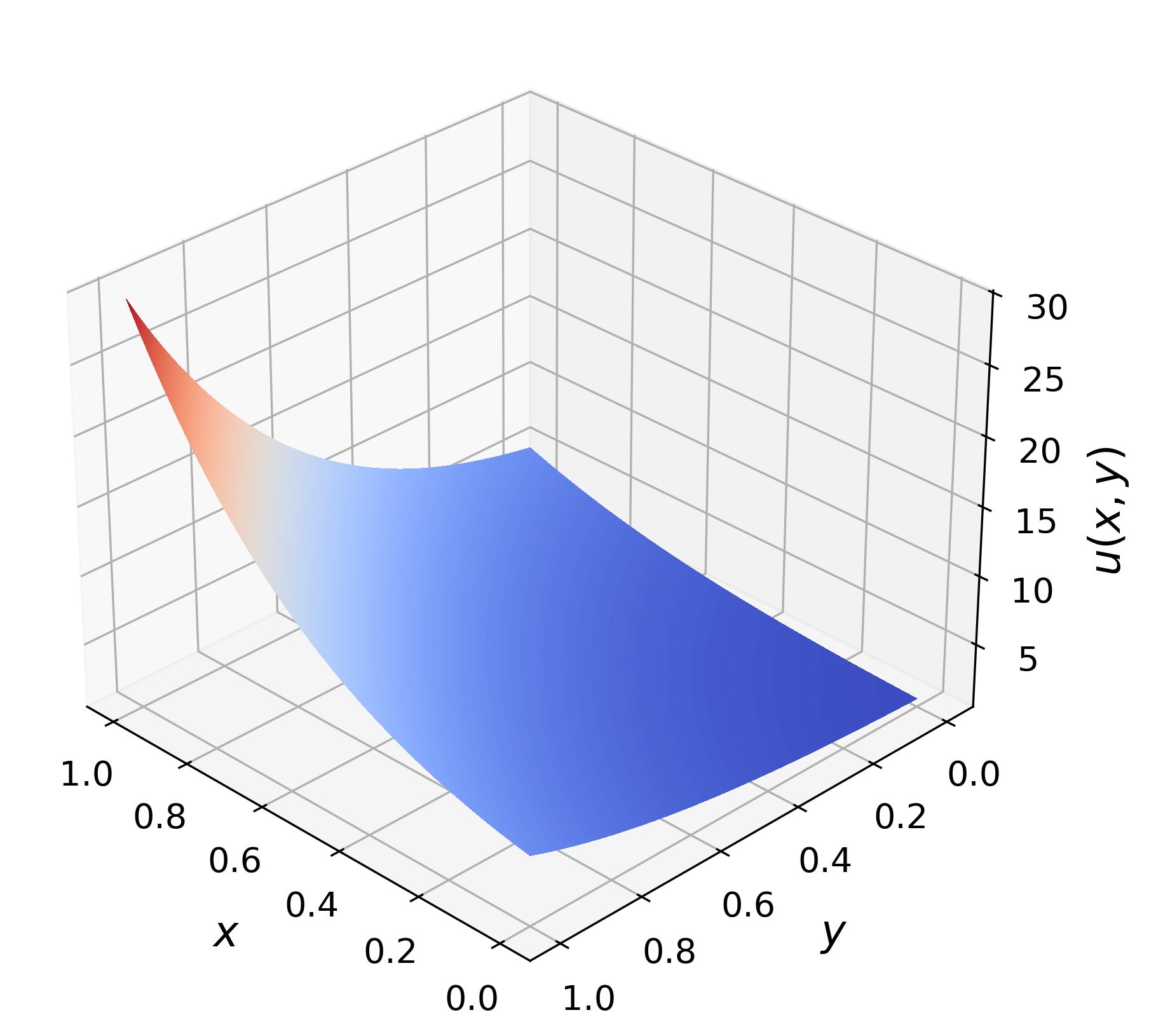}} 
{\includegraphics[width=0.32\textwidth]{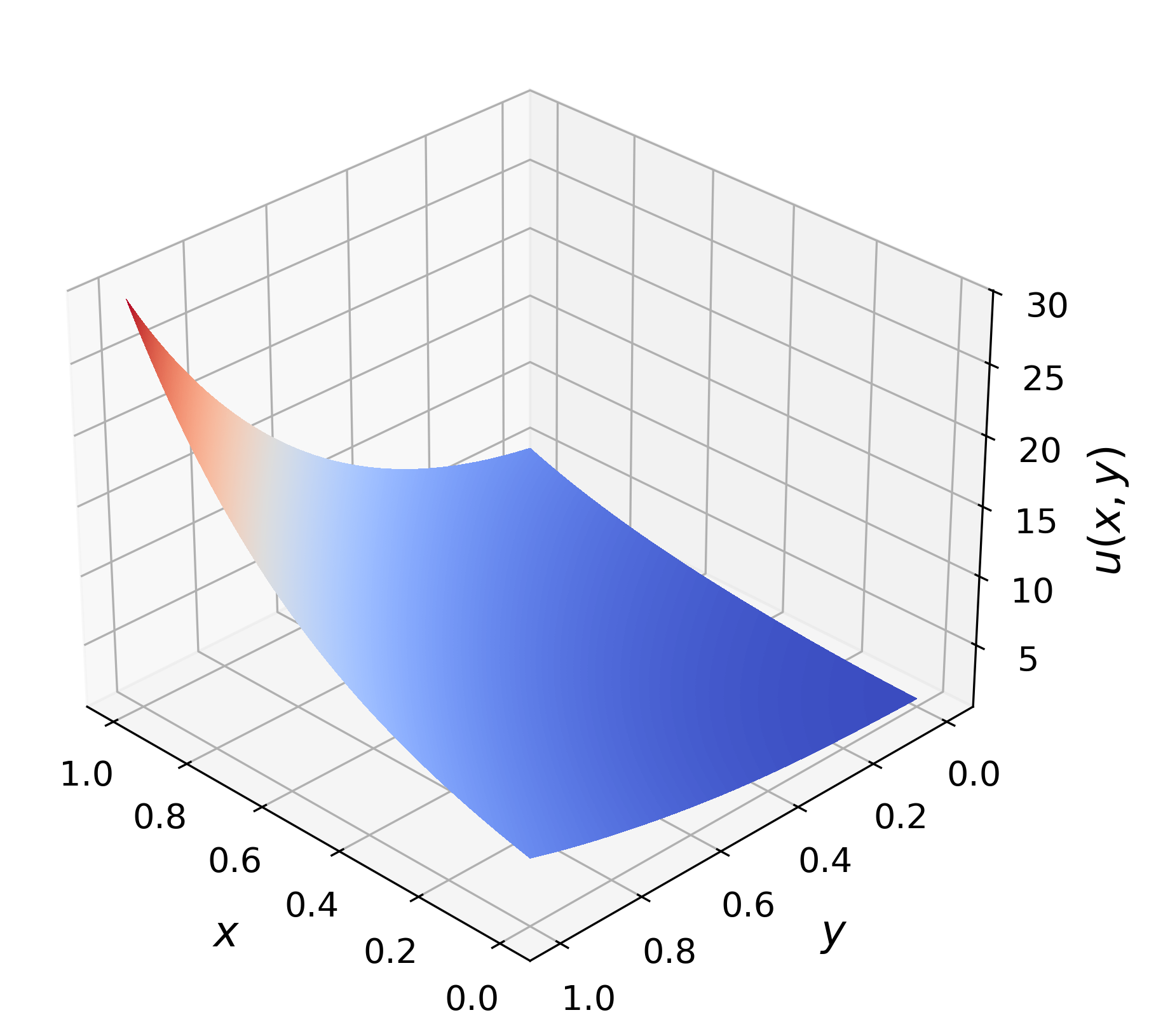}}
{\includegraphics[width=0.32\textwidth]{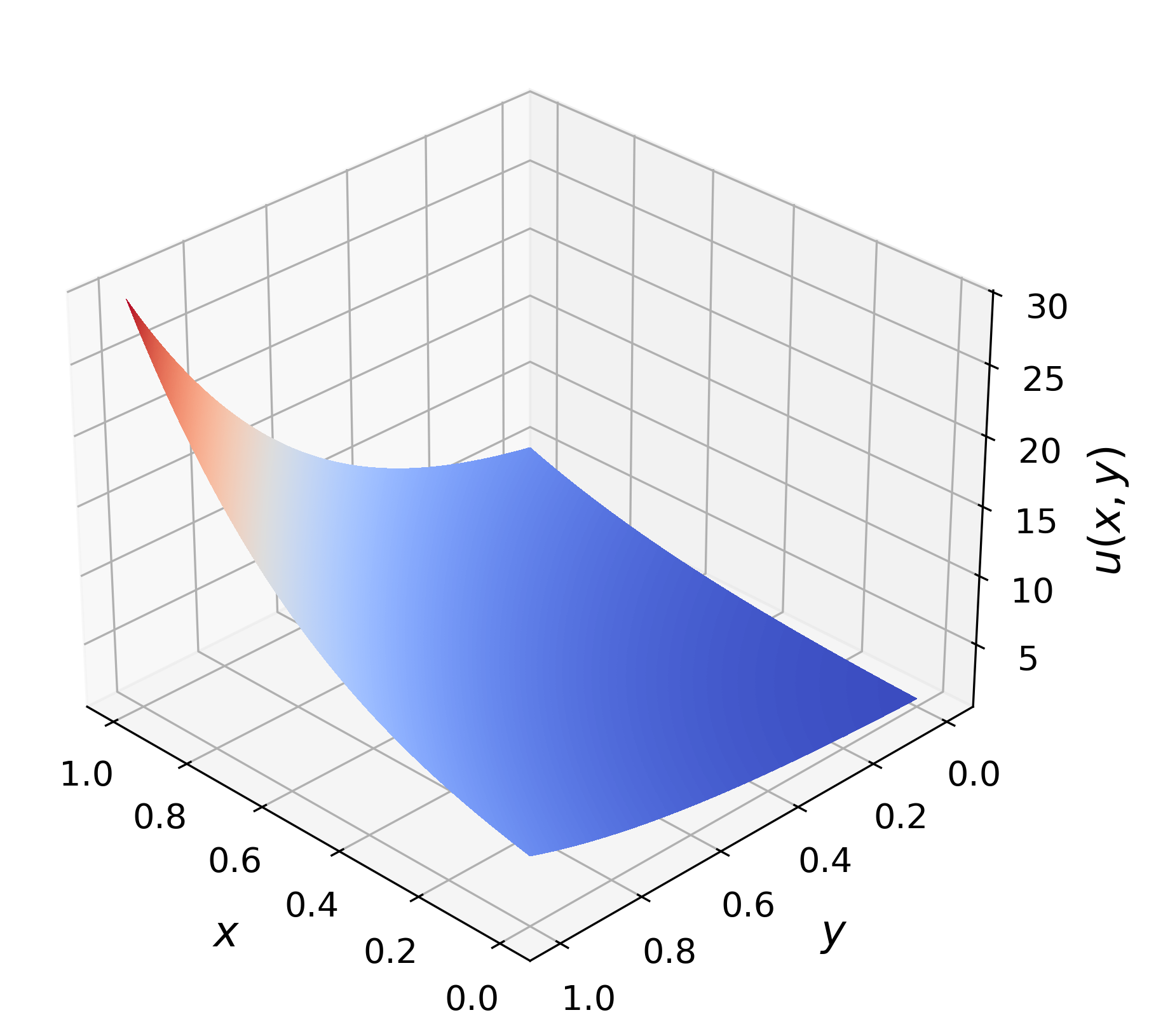}}
}

\caption{Visualization of the true solution (left), standard PINNs prediction using $\L^p_{M}$ (middle) and CPINNs prediction using $\L_{M}^{c}$ (right), evaluated on $15 \times 15$ grid for $\boldsymbol{u_5}$ on $\nu = 0.4$(top row) and $\nu = 0.49$(bottom row).}
\end{figure}

\subsection{Experiment 6 (Biharmonic obstacle problem)} 
We consider a two-dimensional biharmonic obstacle problem with a smooth manufactured solution and a scalar obstacle. The problem is defined as
\begin{align}\label{AN8}
\begin{cases}
\Delta^2 u \geq f \quad &\text{in } \Omega= (-1,1)^2, \\
u \geq \psi,\, (u-\psi)(\Delta^2 u - f) = 0  \quad &\text{in } \Omega, \\
u = g_1,\ \Delta u = g_2 \quad &\text{on } \partial\Omega.
\end{cases}
\end{align}
Introducing the auxiliary variable $v= -\Delta u$, the fourth-order problem is rewritten as the second-order system
\begin{align*}
\Delta v \geq f \quad &\text{in } \Omega, \\
-v + \Delta u = 0 \quad &\text{in } \Omega,
\end{align*}
subject to the constraint $u\geq \psi$, the complementarity condition $(u-\psi)(-\Delta v-f)=0$ in $\Omega$, and the boundary conditions $u=g_1, v=g_2$ on $\partial\Omega$. We prescribe the exact solution and obstacle as
\begin{align*}
u_6 = \cos x\cos y,\ \psi_6 = 0.6\,\cos x\cos y \ \text{on}\ \Omega.
\end{align*}
The forcing term is defined pointwise as
\[
f(x,y):=
\begin{cases}
4\cos x \cos y, & \cos x \cos y \ge 0.6,\\[4pt]
0, & \text{otherwise}.
\end{cases}
\]
Numerical errors and visualizations of the computed solution are reported below.

\subsection{Experiment 7}

In this experiment, we consider an obstacle problem with a constant obstacle on the domain $\Omega = (-1,1)^2.$ The obstacle function is chosen as $\psi_{7}(x,y) \equiv 0.$ For a fixed parameter \( r<1 \) and $|\boldsymbol{x}|^2 = x^2+y^2$, the forcing term is defined by
\[
f_{7}(x,y) =
\begin{cases}
-4\big(2|\boldsymbol{x}|^2 + 2(|\boldsymbol{x}|^2 -r^2)\big), & |\boldsymbol{x}|> r,\\[1mm]
-8r^2\big(1-(|\boldsymbol{x}|^2 -r^2)\big), & |\boldsymbol{x}| \le r.
\end{cases}
\]
The Dirichlet boundary condition is prescribed by $g_{7}(x,y) = (|\boldsymbol{x}|^2 -r^2)^2.$ With these choices, the exact solution of the obstacle problem is given by $$u_{7}(x,y) = \bigl(\max\{|\boldsymbol{x}|^2 -r^2,\,0\}\bigr)^2,$$ which exhibits a nontrivial coincidence set corresponding to the disk $\{|\boldsymbol{x}| \le r\}$. This example serves as a benchmark problem with a known analytical solution and is used to assess the accuracy and convergence properties of the proposed PINNs formulations.
\begin{table}
\centering
\renewcommand{\arraystretch}{1.1}
\begin{tabular}{|c|c c c c|c|}
\hline
Grid &\multicolumn{4}{c|}{$u_{7}(x,y) $} \\
\cline{2-5}
points & Error $(\%)$& Loss $(10^{-5})$  & $\mathfrak{R}_{\mathcal{U}}(\tilde{m}, \bar{m})$ & $ \|(v,\mu)\|_{\U}^2$ \\
$N$ & PINNs / CPINNs & $\L^p_{M}$ / $\L_{M}^c$ & 	  & PINNs / CPINNs  \\
\hline
$5$ & 7.02 / 4.57 & $0.5$ / $0.3$ &  $3.90 \times 10^{-3}$ & $96.91$ / $96.92$ \\
$10$ & 5.17/ 4.63 & $1.0$ / $1.0$  &  $7.71\times 10^{-4}$  & $94.75$ / $94.80$ \\
$15$ & 3.11/ 1.02 & $1.1$ / $3.7$  & $3.18 \times 10^{-4}$ & $94.71$ / $94.72$ \\
$20$ & 3.23 / 1.12& $3.3$ / $1.2$  & $1.73 \times 10^{-4}$  & $94.82$ / $93.89$ \\
$25$ & 4.23 / 1.67& $2.9$ / $1.8$  & $1.08 \times 10^{-4}$  & $89.01$ / $94.67$ \\
$30$ & 3.18/ 1.13& $3.2$ / $1.7$ & $7.43 \times 10^{-5}$ & $93.72$ / $93.92$ \\
\hline
\end{tabular}
\caption{Error comparison of solution \(u_{8}\) obtained by using CPINNs and standard PINNs.}
\label{table_6}
\end{table}

\begin{figure}
\subfloat{
{\includegraphics[width=0.32\textwidth]{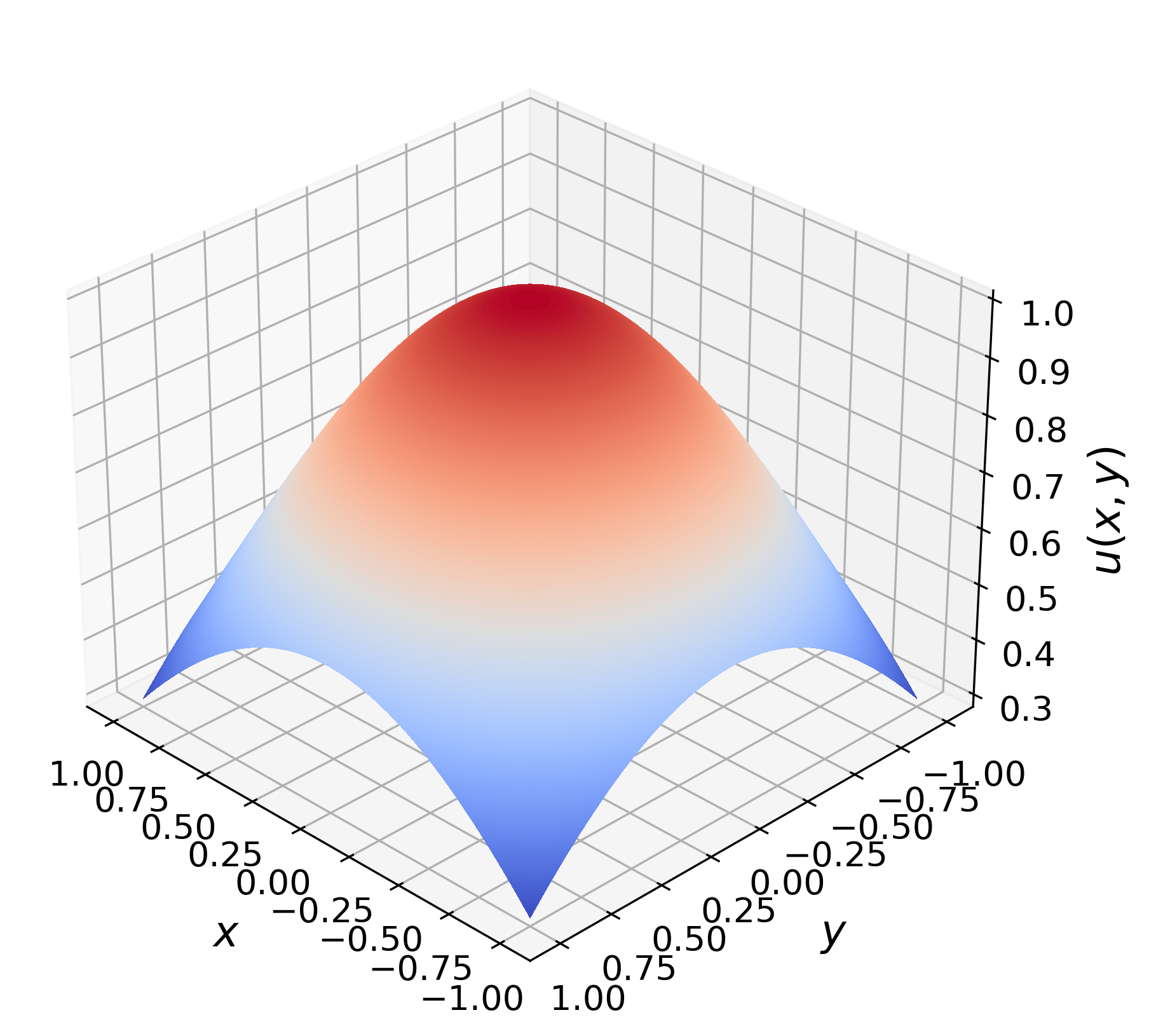}} 
{\includegraphics[width=0.32\textwidth]{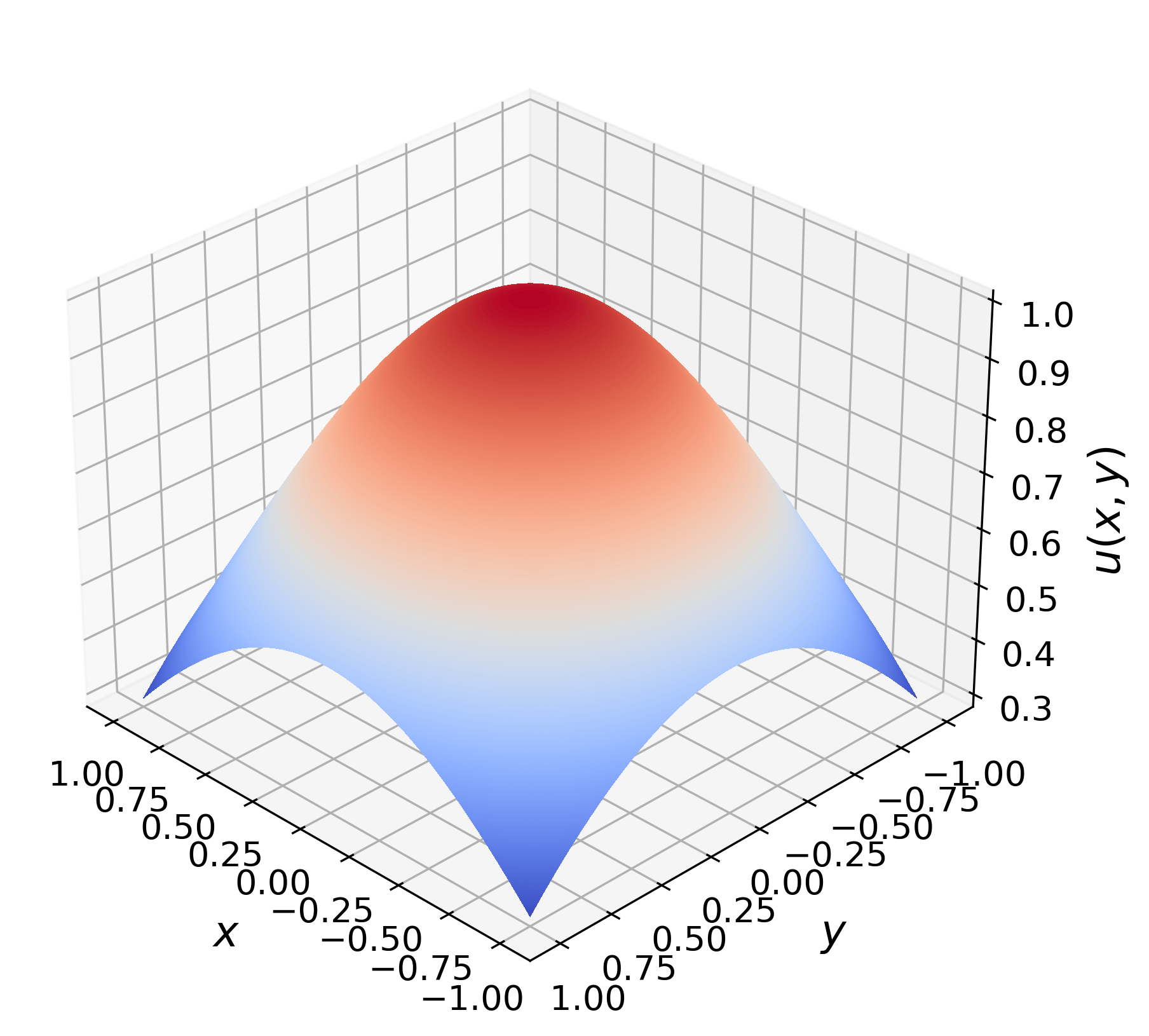}}
{\includegraphics[width=0.32\textwidth]{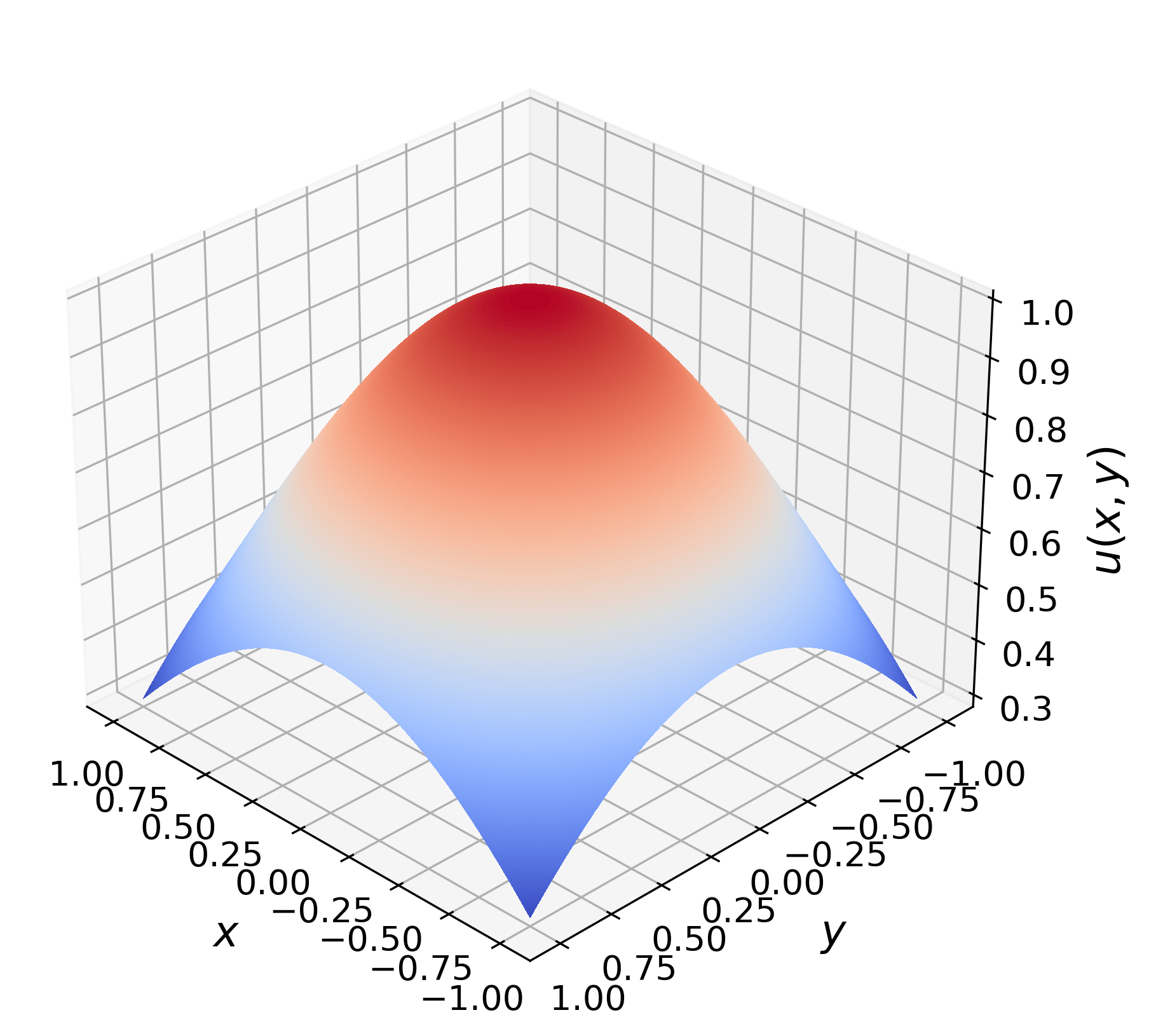}}
}\\
\subfloat{
{\includegraphics[width=0.32\textwidth]{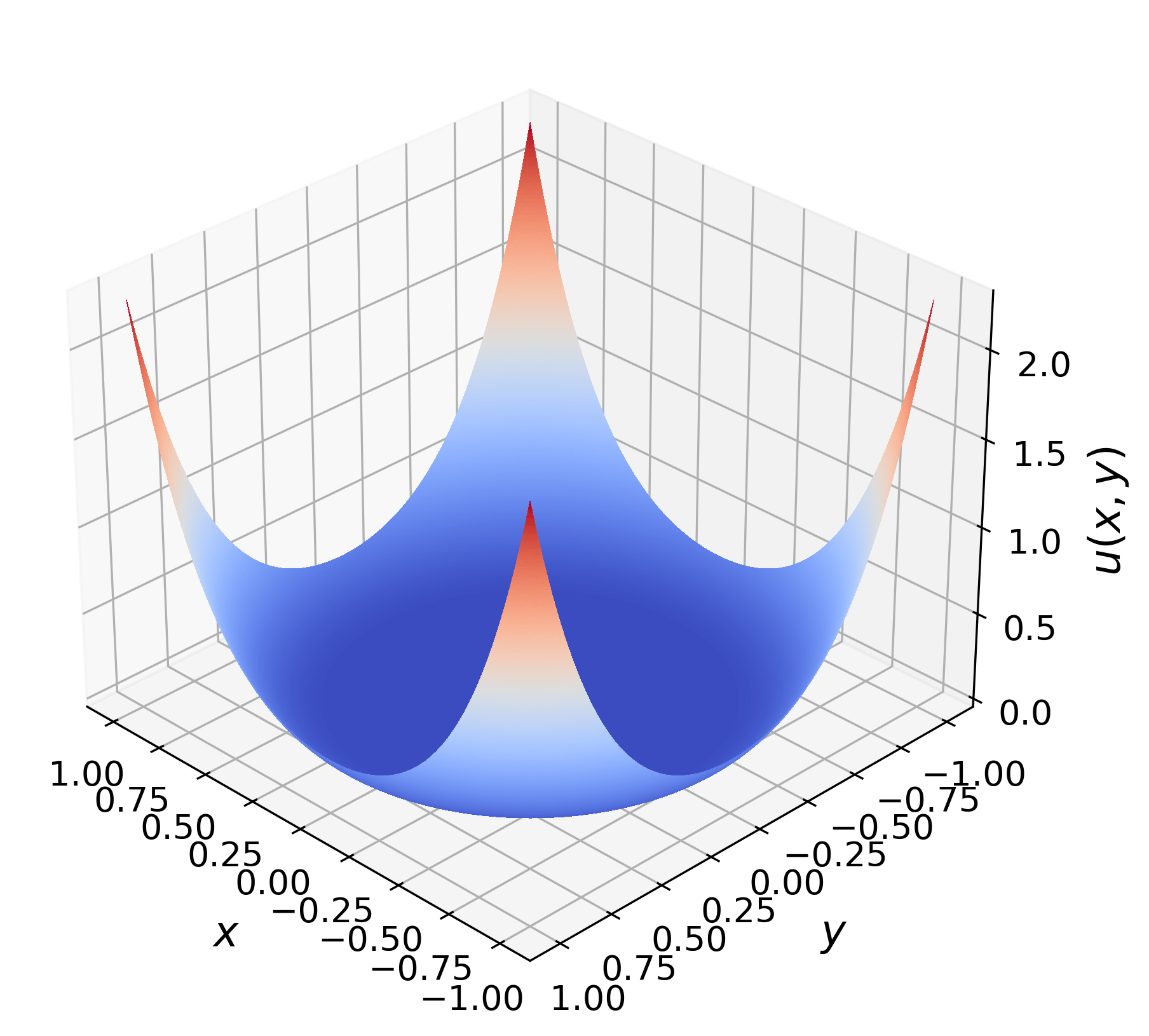}} 
{\includegraphics[width=0.32\textwidth]{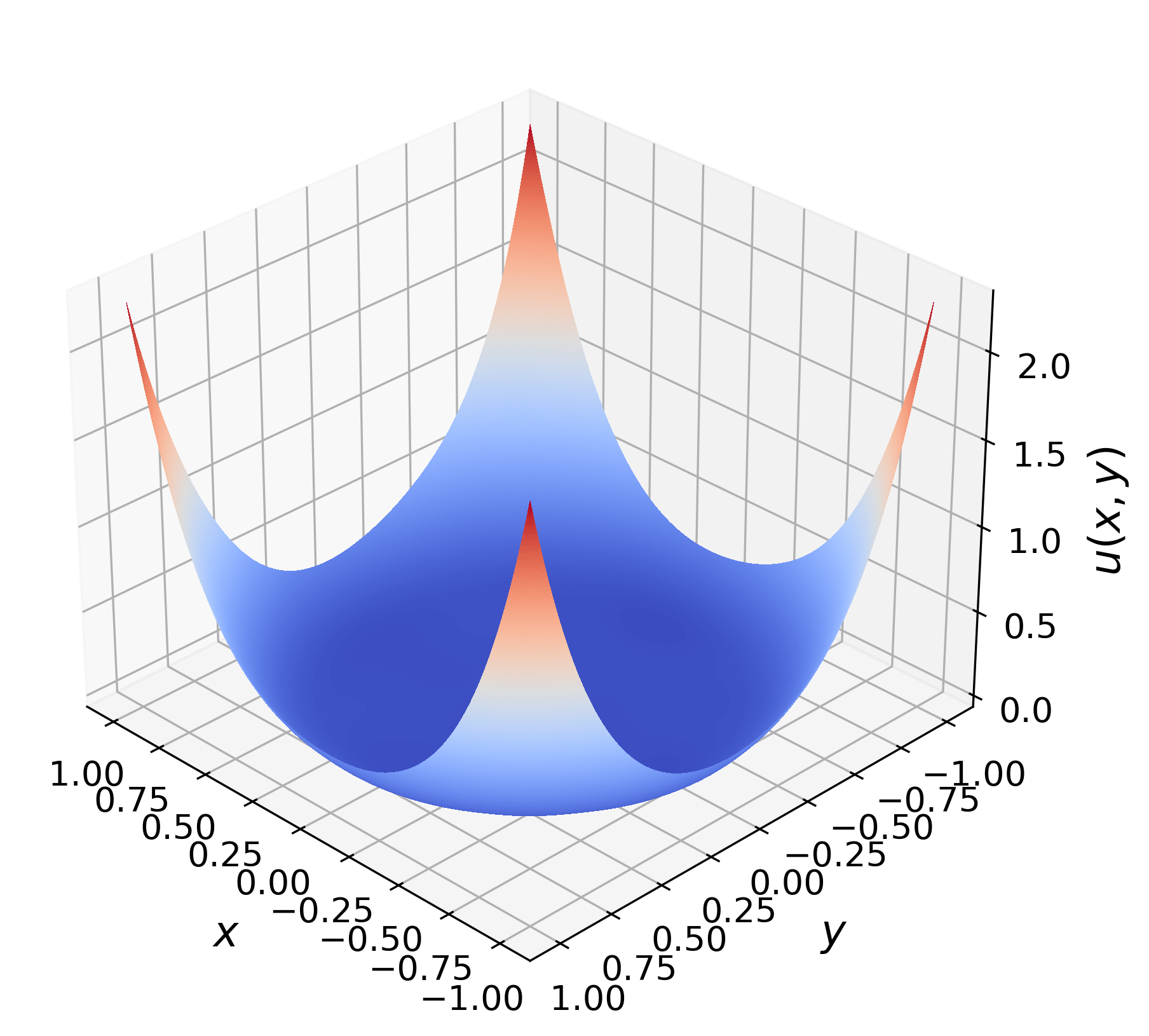}}
{\includegraphics[width=0.32\textwidth]{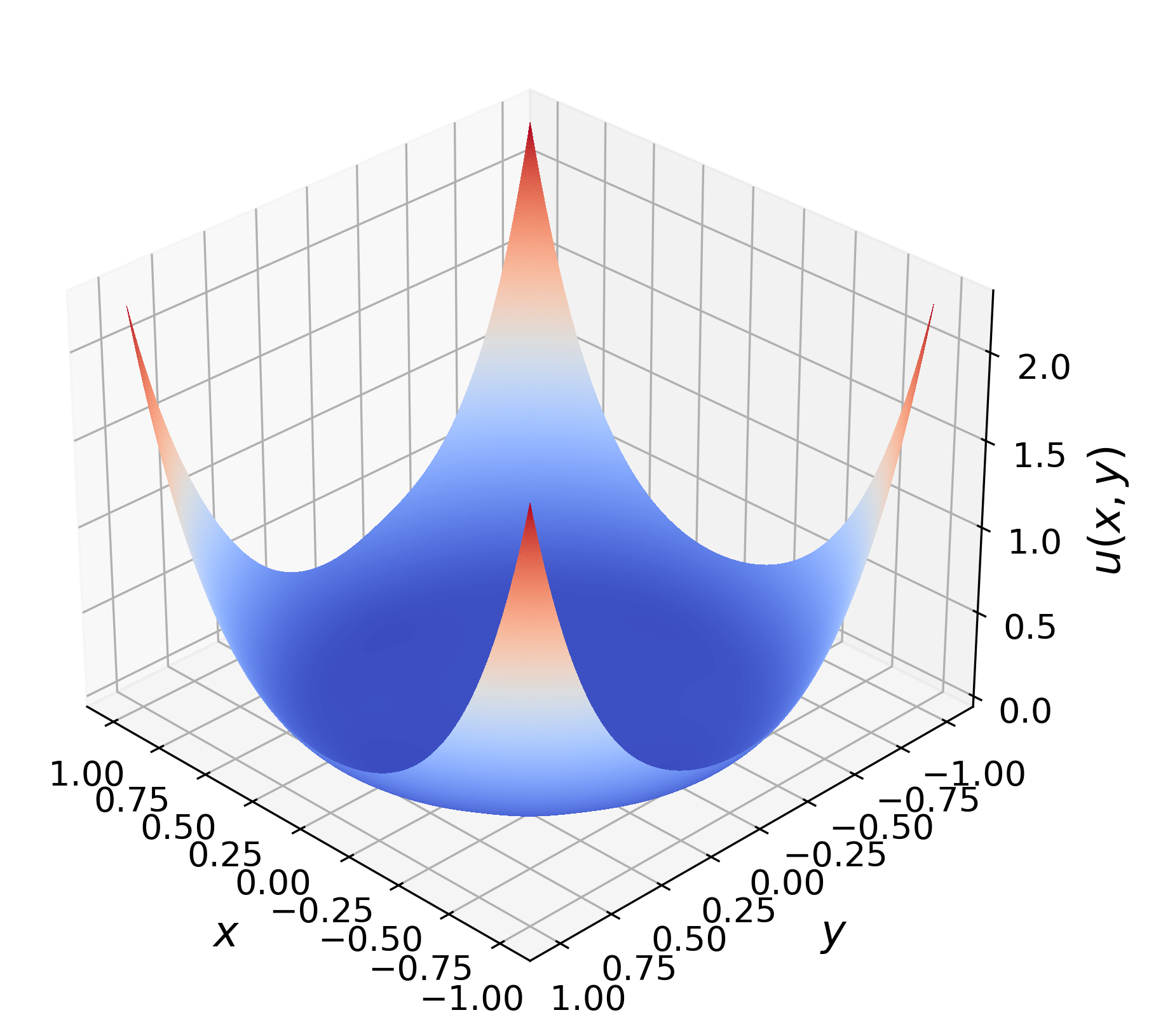}}
}\\

\caption{Plots of the true solution (left), standard PINNs prediction using $\L^p_{M}$ (middle) and CPINNs prediction using $\L_{M}^{c}$ (right), evaluated on $15 \times 15$ grid for for $u_6$ (top row) and $u_7$ (bottom row).}
\end{figure}
The numerical results reported in Table~\ref{table_6} confirm the theoretical estimate
\begin{equation}\label{eq:exp8_stability}
\|(u,\lambda)-(v,\mu)\|_{U}^{2}
\;\lesssim\;
\L_{M}^{c}(v,\mu)
+
\big(\|(v,\mu)\|_{\U}^{2}+1\big)\,
\mathfrak{R}_{\U}(\tilde m,\bar m),
\end{equation}
where the recovery residual \(\mathfrak{R}_{\U}(\tilde m,\bar m)\) is evaluated for the
Besov smoothness parameter \(s=2\).
As the number of interior and boundary collocation points increases, both the
discrete mixed loss \(\L_{M}^{c}\) and the recovery residual \(\mathfrak{R}_{\U}(\tilde m,\bar m)\)
decrease, which results in a consistent reduction of the \(\U\)-norm error observed in
the table.

\section{Conclusion}
This work develops a consistent physics-informed neural network framework for obstacle problems, grounded in PDE stability and variational inequality theory. By formulating mixed CPINNs loss equivalent to the natural stability norms, we establish rigorous optimal recovery guarantees for both the solution and the associated Lagrange multiplier under Besov regularity assumptions. Discrete norm equivalence results have been established, which further yield fully computable training losses based on pointwise data, which bridges the gap between continuous theory and practical implementation. Numerical experiments based on scalar, vector-valued, and higher-order obstacle problems confirm the accuracy, robustness, and superiority of the proposed approach over standard PINNs. These results extend CPINNs theory beyond classical PDEs to inequality-constrained problems and provide a solid foundation for future research on time-dependent obstacle problems, adaptive sampling strategies, and broader classes of variational inequalities.

\section*{Data availability}
The datasets generated during this study are available from the corresponding author upon reasonable request.

\bibliographystyle{siamplain}
\bibliography{reference_2n}
\end{document}